\documentclass[reqno]{amsart}

\usepackage[T1]{fontenc}
\usepackage[utf8]{inputenc}

\usepackage{amssymb}
\usepackage{enumitem}
\usepackage{hyperref}
\usepackage{cleveref}
\usepackage{stmaryrd}
\usepackage{tikz}

\newtheorem{thm}{Theorem}[section]
\newtheorem{lem}[thm]{Lemma}
\newtheorem{cor}[thm]{Corollary}
\newtheorem{prop}[thm]{Proposition}

\theoremstyle{definition}
\newtheorem{defin}[thm]{Definition}
\newtheorem{ex}[thm]{Example}

\theoremstyle{remark}
\newtheorem{rk}[thm]{Remark}

\numberwithin{equation}{section}

\def\N{{\mathbb N}}    
\def\R{{\mathbb R}}    
\newcommand{\eps}{\varepsilon} 
\renewcommand{\phi}{\varphi} 
\newcommand{\libr}{\llbracket} 
\newcommand{\ribr}{\rrbracket} 
\newcommand{\tild}[1]{\tilde{#1}}
\newcommand{\inter}[1]{\text{int}(#1)}
\newcommand{\inner}[2]{\langle #1,#2 \rangle}

\newcommand{\intef}[1]{\int_{\gamma^{#1}}{f\,ds}}
\newcommand{\tub}[1][\eps_1]{\mathcal{N}_{#1}}
\newcommand{\tubb}[1][\eps_1]{\mathcal{N}^{*}_{#1}}

\newcommand{\tps}[1][e,p]{t_{#1}}
\newcommand{\sps}[1][e,p]{s_{#1}}
\newcommand{\bdry}{\partial M}
\newcommand{\ray}{\mathcal{I}f}

\title{Reconstruction of piecewise constant functions from X-ray data}
\date{\today}

\author[V. Lebovici]{Vadim Lebovici}
    \email{vadim.lebovici@ens.fr}
    \address{Département de mathématiques et applications, Ecole Normale Supérieure, 45, rue d'Ulm, 75005 Paris, France}
    \curraddr{}

\thanks{The author would like to first thank Mikko Salo for bringing this problem to his attention, for having closely followed its resolution and for their helpful discussions. The author is also thankful to the anonymous referees for useful suggestions.}

\begin{document}

   \begin{abstract}
       We show that on a two-dimensional compact nontrapping Riemannian manifold with strictly convex boundary, a piecewise constant function can be recovered from its integrals over geodesics. We adapt the injectivity proof which uses variations through geodesics to recover the function and we improve this result when the manifold is simple and the function is constant on tiles with geodesic edges, showing that the Jacobi fields of these variations are sufficient. We give also explicit formulas for the values near the boundary. We finally study the stability of the reconstruction method.
   \end{abstract}
   
    \maketitle
   
    \tableofcontents
    
    \section{Introduction}
        X-ray tomography consists in finding the attenuation coefficient at every point of a non-homogeneous medium from the given attenuation of every ray of light passing through it. If $(M,g)$ is a compact Riemannian manifold with boundary, the X-ray transform $\ray$ of a function $f : M \to \R$ is the collection of the integrals of $f$ over all geodesics joining boundary points. The mathematical formulation of X-ray tomography is the inverse problem consisting in the recovering of $f$ from the knowledge of $\ray$. First introduced by Fritz John in 1938 \cite{F38}, X-ray transform is one of the cornerstones of geometric inverse problems, arising for instance in the Calderón problem \cite{CALDERON2006,SU87} or the boundary and scattering rigidity problems \cite{Paternain2014,Uhlmann2014}.

The idea of recovering information on a medium from measurements at its boundary gave rise to various non-invasive imaging methods, brightly summed up in \cite{ilmavirta2018integral}. As a short overview, X-ray computerized tomography allowed non-invasive medical imaging and non-destructive volumetric study of rare specimens (fossils, meteorites) \cite{schwarz2005neutron}, seismic tomography is of use in seismology to reconstruct the density inside the Earth \cite{herglotz1905uber,wiechert1907erdbebenwellen}, acoustic tomography is used in ocean imaging \cite{munk1979ocean} and neutron spin tomography \cite{sales2018three} in mineralogy and geochemistry \cite{WinklerBjoern}. 

Three questions arise in the X-ray tomography in order: first the injectivity of $\mathcal{I}$, second whether or not $f$ can be recovered from the knowledge of $\ray$ and finally the stability of such a reconstruction.

The injectivity has already been studied for simple manifolds by Mukhometov \cite{Mukho77}, Anikonov \cite{Anikonov2009} and Sharafutdinov \cite{sharafutdinov2012integral}. In dimension strictly greater than two, injectivity has been proved without the assumption that the manifold is simple by Uhlmann and Vasy in \cite{Uhlmann2016}. See also \cite{ilmavirta2018integral, Paternain2014} for a more comprehensive list of results. It has been conjectured by Paternain, Salo and Uhlmann \cite{paternainsalouhlmann2013} that the X-ray transform is injective on compact nontrapping Riemannian manifolds with strictly convex boundary. This conjecture has been proved by Ilmavirta, Lehtonen and Salo \cite{ILS} for the simpler tomography of piecewise constant functions (see \cref{def:piecewise constant function}). Here we say that the Riemannian manifold $(M,g)$ has \emph{strictly convex boundary} if the second fundamental form of $\bdry$ in $M$ is positive definite and is \emph{nontrapping} if for any $(p,w)\in TM$, the geodesic starting from $p$ with initial velocity $w$ meets the boundary in finite time.

The work in \cite{ILS} proves injectivity for measurements related to infinitely many geodesics. Once injectivity is known, due to finite-dimensionality the problem reduces to inverting a matrix that corresponds to finitely many geodesics, and the inverse map is automatically Lipschitz continuous. However, this abstract argument does not give any information on which geodesics one should use for reconstruction, or any estimates for the stability constants. These questions will be addressed in the two-dimensional case in this article.

\subsection{Reconstruction}

In the Euclidean case, Radon \cite{radon20051} proved an inversion formula for the X-ray transform and Helgason \cite{helgason1990totally} proved it in the hyperbolic case. Pestov and Uhlmann established approximate reconstruction formulas on simple compact surfaces with strictly convex boundary in \cite{PU04} and Krishnan proved in \cite{krishnan2010inversion} that these formulas can be made exact in a neighborhood of metrics with constant curvature. These formulas has been implemented by Monard in \cite{monard2014numerical} and have been also generalized in negative curvature by Guillarmou and Monard in \cite{GM15}. 

In this paper, we give a positive answer to the reconstruction question for piecewise constant functions on compact nontrapping Riemannian two-dimensional manifolds with strictly convex boundary. For piecewise constant functions, the reconstruction method does not use heavy machinery.

\begin{thm}\label{main:noass}
    Let $(M,g)$ be a two-dimensional compact nontrapping Riemannian manifold with strictly convex boundary and $f : M\to\R$ be a piecewise constant function on a regular tiling. Then $f$ can be recovered from the knowledge of the metric, the tiling and $\ray$.
\end{thm}

We give two proofs of this fact in this article. The first one is a direct adaptation of the injectivity proof of \cite{ILS}. To iterate the argument at level sets of a strictly convex foliation inside the manifold, the proof relies on the local reconstruction \cref{lem:reconstruction noass iteration step} which uses specific variations through geodesics to recover the unknown values. The second proof is an improvement of the first one if $M$ is assumed to be simple and if the tiling on which $f$ is piecewise constant has geodesic edges (see \cref{def:geodesic tiling}). This new proof relies on the improved local reconstruction \cref{lem:reconstruction simplegeod iteration step} which is more elegant since it uses less knowledge on the tiling and only the Jacobi fields of the previous variations through geodesics to recover the function.

We say that a Riemannian manifold is \emph{simple} if it is simply connected, there are no conjugate points and the boundary is strictly convex.

\subsection{Stability}

The first injectivity proofs by Mukhometov gave some stability estimates with a loss of a 1/2 derivative and the stability of the reconstruction has been studied on simple surfaces by Sharafutdinov in \cite{sharafutdinov1999ray}. A sharp $L^2\to H^{1/2}$ stability estimate has also been proved recently by Assylbekov and Stefanov in \cite{assylbekov2018sharp} under the same assumptions.

We study the stability of the previous reconstruction method. X-ray tomography being a linear inverse problem, a stability statement consists here in bounding by above a norm of $f$ by a norm of $\ray$. As we said previously, the theoretical argument for existence of a Lipschitz continuous inverse does not give any information on the Lipschitz constant.

In this paper, we give more elementary and we make explicit stability estimates for the reconstruction method used in the simple case of piecewise constant functions. Our method implies taking derivatives of restricted integrals, see \cref{sctn:corner} and \cref{subsctn:proof reconstruction noass}. As we will see, the lack of regularity of $f$ implies a lack of regularity of the ray transform and we cannot put a suitable norm on it to control these derivatives, see \cref{sctn:regularity}. We will therefore state stability estimates using norms of the restrictions of the integrals of $f$ along portions of geodesics considered in the reconstruction because these parts of integrals are smooth with respect to the parameters. These restricted integrals are not initially given in the problem, but they can be computed from the knowledge of the ray transform, the metric and the tiling.

\subsection{Discussions}

First, one could wonder at this point whether the assumption of \emph{a priori} knowledge of the tiling is reasonable or not. The tiling on which the function is piecewise constant could be known in practice: in numerical implementations, the tiling corresponds to the pixel grid used to represent $f$, see for instance \cite[Section 3]{monard2014numerical} and \cite[Section 2.2]{monard2019efficient}. Using the reconstruction technique of this article would lead to reconstruct a piecewise constant approximation of $f$ on this known grid.

Second, if the tiling is unknown, one could wonder whether recovering it is feasible or not. On simple manifolds, one may use the previously cited existing explicit reconstruction formulas for generic functions to avoid the need of the tiling but these techniques use heavier machinery than the reconstruction exposed here. To the best knowledge of the author, there is no results on recovering a tiling on which an unknown function is piecewise constant from X-ray data in general. One may expect that this would be possible, with elementary geometric methods, but this has no particular reason to be easy and it will be the subject of future work of the author. One may also use other type of measurements to recover the tiling if possible. In fact, in other inverse problems such as the recovering of a medium from far-field pattern, one may recover the values of a potential if it is piecewise constant on a known polyhedral cell geometry but also on a possibly unknown nested polyhedral geometry as shown in \cite{blaasten2017recovering}. 

\subsection{Outline}

The method for the reconstruction is the same as for the injectivity in \cite{ILS}. We first solve the Euclidean case in \cref{sctn:corner} and we show that near any point of $M$ the problem at a corner can be reduced to an analogue in the Euclidean plane. Then, we reconstruct $f$ near $\bdry$ in \cref{sctn:boundary} with explicit formulas. In \cref{sctn:iteration}, we state a local reconstruction lemma around a point and iterate it in all the manifold thanks to the strictly convex foliation. Finally, we improve the local reconstruction lemma in \cref{sctn:simplegeod} under the assumptions that the manifold is simple and the tiling is geodesic. The stability of the reconstruction method is studied in \cref{sctn:stability}.
    \section{Preliminaries}
        In this section, let $(M,g)$ be a smooth Riemannian $n$-manifold with or without boundary. 
We define piecewise constant functions and foliations as in \cite{ILS}. 


\subsection{Regular tilings and piecewise constant functions}

\begin{defin}[Regular simplex]
    We call regular $n$-simplex on $M$ the image of a $C^\infty$-embedding of the standard $n$-simplex from $\R^{n+1}$ to $M$.
\end{defin}

\begin{defin}[Depth of a point in a regular simplex]
    We define the depth of any point in a regular simplex by induction: interior points have depth $0$, the interiors of the regular $(n-1)$-simplices making up the boundary have depth $1$, the interiors of their boundary simplices of dimension $n-2$ have depth $2$ and so on. Finally, the $n+1$ corner points have depth $n$.
\end{defin}

\begin{defin}[Regular tiling]\label{def:regular tiling}
    A regular tiling of $M$ is a collection of regular $n$-simplices $(\Delta_i)_{i\in I}$ such that,
    \begin{enumerate}[label=(\roman*)]
        \item the collection is locally finite: for any compact subset $K\subset M$ the set $\{i\in I, \Delta_i\cap K\ne \varnothing\}$ is finite,
        \item $M=\bigcup_{i\in I}\Delta_i$,
        \item $\inter{\Delta_i}\cap\inter{\Delta_j}=\varnothing$ when $i\ne j$ and
        \item If $x\in\Delta_i\cap\Delta_j$, then $x$ has the same depth in both $\Delta_i$ and $\Delta_j$.
    \end{enumerate}
\end{defin}

The tiles of a regular tiling have boundary simplices which have boundary simplices and so on. In our two-dimensional study, a regular tiling of $M$ will thus be a collection of triangles which can intersect each other only at common vertices or all along a common edge. Since the collection is locally finite, if $M$ is compact, a regular tiling has a finite number of tiles.

\begin{defin}[Piecewise constant function]\label{def:piecewise constant function}
    We say that a function $f:M\to\R$ is piecewise constant if there exist a tiling $(\Delta_i)_{i\in I}$ such that $f$ is constant on the interior of each regular $n$-simplex $\Delta_i$ and vanishes on $\bigcup_{i\in I}\partial\Delta_i$.
\end{defin}

\subsection{Tangent data}


The reconstruction method relies on the reduction of the problem near a point to the reconstruction of a conical function (see \cref{subsctn:conical functions}) in the Euclidean plane. To do so, we define here analogues in the tangent plane of tiles and functions which we will call tangent cones and tangent functions.

\begin{defin}[Tangent cones of a regular simplex]
    Consider a regular $m$-simplex $\Delta$ in $M$ with $0\leq m\leq \dim(M)$. Let $p\in \Delta$ and $\mathcal{C}$ be the set of all $C^\infty$-curves starting at $p$ and staying in $\Delta$. The tangent cone of $\Delta$ at $p$, denoted by $C_p\Delta$, is the set
    \begin{equation*}
        \{\dot\gamma(0),\gamma\in \mathcal{C}\}\subset T_pM.
    \end{equation*}
\end{defin}

\begin{defin}[Tangent function]
    Let $f:M\to\R$ be a piecewise constant function and $p\in M$. Let also $\Delta_1,...,\Delta_N$ be the tiles that contains $p$ and $a_1,...\,,\,a_N$ the respective values of $f$ on those tiles. The tangent function of $f$ at $p$ is the function $T_pf:T_pM\to\R$ defined by
    \begin{equation*}
        T_pf(u)=
        \begin{cases}
        a_i & \text{ if } u\in\inter{C_p\Delta_i}\\
        0 & \text{ if } u\not\in\bigcup_{i=1}^N C_p\Delta_i
        \end{cases}
        .
    \end{equation*}
\end{defin}

\subsection{Parametrization of unit tangent vectors}
At a point $p\in M$, if $(\omega,\nu)$ is an orthonormal basis of $T_pM$, we parametrize a unit tangent vector $w$ at $p$ by its angle $\theta\in(-\pi,\pi]$ in polar coordinates in $T_pM$ and denote $w_\theta$ the unit tangent vector at $p$ defined by the angle $\theta$.

\begin{rk}
	Note also that for convenience we will systematically omit to write an index $\theta$ when $\theta=0$. Therefore $w$ will denote $w_0$, $\gamma$ will denote $\gamma_0$ and so on.
\end{rk}

\subsection{Foliations}\label{subsctn:foliation}

As shown in \cite{BGL02}, if $M$ is a two-dimensional compact nontrapping Riemannian manifold with strictly convex boundary there exists a strictly convex function $\varphi : M\to \R$. Roughly, this strictly convex function is constructed by letting the strictly convex boundary evolve under the mean curvature flow, so in particular the boundary is a level set. 

The existence of a strictly convex function implies that $M$ is foliated by the level sets $\{\phi=c\}$ for $\min\varphi< c\leq\max\varphi$ which are strictly convex hypersurfaces for the normal pointing toward the set $\{\phi\leq c\}$. Moreover, a geodesic in a manifold with strictly convex foliation $\phi$ is tangential to a level set of $\phi$ at at most one point (see the proof of \cref{lem:geodesics in foliated manifolds} and figure \ref{fig:geodesics in foliated manifolds}). There is such a point
if and only if the geodesic does not go through the minimum of $\phi$. For more details on strictly convex functions and foliations, see \cite{PSUZ16, UD94}.

\begin{figure}[h!]
\begin{tikzpicture}[scale=0.4]

	\draw (0,0)ellipse(6 and 5);
	
	\draw [dashed,scale=0.7] (0,0)ellipse(6 and 5);
	\draw [dashed,scale=0.45] (0,0)ellipse(6 and 5);
	\draw [dashed,scale=0.2] (0,0)ellipse(6 and 5);
	
	\coordinate (A) at (-5.25,-2.42);
	\coordinate (B) at (2.05,-4.7);
	\coordinate (P) at (-1.03,-2.08);
	
	\draw (A) node[left] {$\gamma$};
	\draw (4.42,3.18) node[left] {$M$};
	\draw (2.7,0) node[scale=0.9] {$\{\phi=c\}$};
		
	\draw (A) .. controls +(38:1) and +(161.07:1.8) .. (P) .. controls +(341.07:1.8) and +(115:1) .. (B);
	
\end{tikzpicture}
	\caption{There is only one point such that the geodesic $\gamma$ in a manifold $M$ is tangential to a level set of the strictly convex function $\phi$.}
	\label{fig:geodesics in foliated manifolds}
\end{figure}
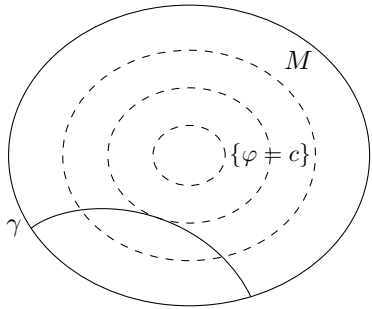

We introduce here some terminology for the different types of tile containing a point on a level set of the foliation which will be useful all along the article. Denote $\Sigma=\{\phi=c\}$, let $p\in\Sigma$ and denote $\nu$ the normal to $\Sigma$ at $p$ pointing toward $\{\phi<c\}$. Denote $H_\pm=\{u\in T_pM,\,\pm \,\inner{\nu}{u} >0\}$ and $H_0=T_p\Sigma$. Near $p$ there are three mutually exclusive types of tiles defined in the proof of \cite[Lemma~5.1]{ILS},
    \begin{enumerate}
        \item simplices $\Delta$ such that $C_p\Delta\cap H_-\ne\varnothing$,
        \item simplices $\Delta$ such that $C_p\Delta\subset H_+\cup H_0$ and $C_p\Delta\cap H_0\ne\{0\}$ which we call \emph{tangent tiles}, and
        \item simplices $\Delta$ such that $C_p\Delta\subset H_+\cup\{0\}$ which we call \emph{corner tiles}.
    \end{enumerate}

\subsection{Curvature and jerk of the boundary}

Let $p\in \bdry$. Let $\nu$ the inward pointing normal to the boundary at $p$ and $\omega \in T_p\bdry$ a unit vector so that $(\omega,\nu)$ is an orthonormal basis of $T_pM$. Denote $U$ a normal neighborhood of $p$ associated to this basis.

\begin{lem}
    \label{lem:parametrization boundary}
    There exist a smooth function $h : \R \to \R$ so that in normal coordinates in $U$,
    \begin{equation*}
        (x,y)\in\bdry \Leftrightarrow y=h(x)
    \end{equation*}
    and 
    \begin{equation*}
        h(x)=\frac{\kappa}{2} x^2 + \frac{j}{6} x^3 + \underset{x\to 0}{o}(x^3)
    \end{equation*}
    where $\kappa>0$ is the curvature and $j:=h^{(3)}(0)$ is called the \emph{jerk} of $\bdry$ at $p$ for $\omega$.
\end{lem}

\begin{rk}
 $ $
	\begin{enumerate}
		\item The jerk of $\bdry$ at $p$ for $\omega$ is uniquely determined and the only non canonical choice is the unit tangent vector $\omega$ which changes the sign of $j$.
		\item We named it \emph{jerk} after the derivative of the acceleration in physics. It describes the variation of curvature of $\bdry$ at $p$. 
	\end{enumerate}
\end{rk}

\begin{proof}
See \cref{subsctn:appendix jerk}.
\end{proof}
      \section{Reconstruction at a corner}\label{sctn:corner}
          From now on, let $(M,g)$ be a two-dimensional compact nontrapping Riemannian manifold with strictly convex boundary and $T$ be a fixed regular tiling of $M$.

In this section, we reconstruct $f$ in a corner. First, we define the corner and the variations used in the reconstruction (see figures \ref{fig:corner reconstruction} and \ref{fig:corner tangent data} for the whole picture). We then state our result in \cref{subsctn:corner result}. To prove it, we first prove a reconstruction result for conical functions in the Euclidean plane. Then, we generalize \cite[Lemma~4.2]{ILS} which is the key lemma to reconstruct $f$ at a corner since it reduces the Riemannian case to the reconstruction of conical functions in the Euclidean plane.

\subsection{Definition of the corner}\label{subsctn:corner def}

Let $\Sigma\subset M$ be a strictly convex hypersurface and $p\in\Sigma$ be a vertex of the tiling. Extend $M$ in $\tild{M}$ by extending the metric around $p$ if necessary so that we can always assume $p\in\inter{\tild{M}}$.

Denote $\nu$ the inward pointing normal to $\Sigma$ at $p$ and take $\omega\in T_pM$ so that $(\omega,\nu)$ is an orthonormal basis of $T_pM$. Here we call inward pointing the normal at $p$ making the scalar second fundamental form of $\Sigma$ strictly positive. 

Denote $\mathcal{C}$ the set of corner tiles $\Delta$ at $p$ and $S$ the sector in $T_pM$ corresponding to $\mathcal{C}$ defined by $S:=\bigcup_{\Delta\in \mathcal{C}} C_p\Delta$.

As in \cite{ILS}, we define the neighborhood in which the reduction to the Euclidean case is valid. Let $r > 0$ be such that $\Sigma$ cuts the ball $B(p,r)$ into two parts and call $V$ the open part for which the part of its boundary coinciding with $\Sigma$ is strictly convex. Reduce $r$ if necessary so that each tile having $p$ for vertex cuts the ball into three parts. Denote finally the corner $C:=\mathcal{C}\cap B(p,r)$.

\subsection{Definition of the variation}\label{subsctn:corner variation}

Take $\theta$ near $0$. Let $\gamma_\theta$ be the geodesic starting from $p$ with initial velocity $w_\theta$ denote $\ell_\theta$ the line in $T_pM$ starting from distance $1$ in the direction $w_{\theta+\pi/2}$ from $p$ with initial velocity $w_\theta$.

Denote $\Gamma : (-\eps_0,\eps_0)\times[a,b] \to M$ a variation through geodesics starting from $\gamma_\theta$, denote $J$ its Jacobi field and $\gamma^\eps_\theta=\Gamma(\eps,\cdot)$. Define the integrals $I_C(\theta,\eps):=\int_{\gamma_\theta^\eps\cap\,C}{f\,ds}$ for $(\theta,\eps)$ in a neighborhood of $(0,0)$.

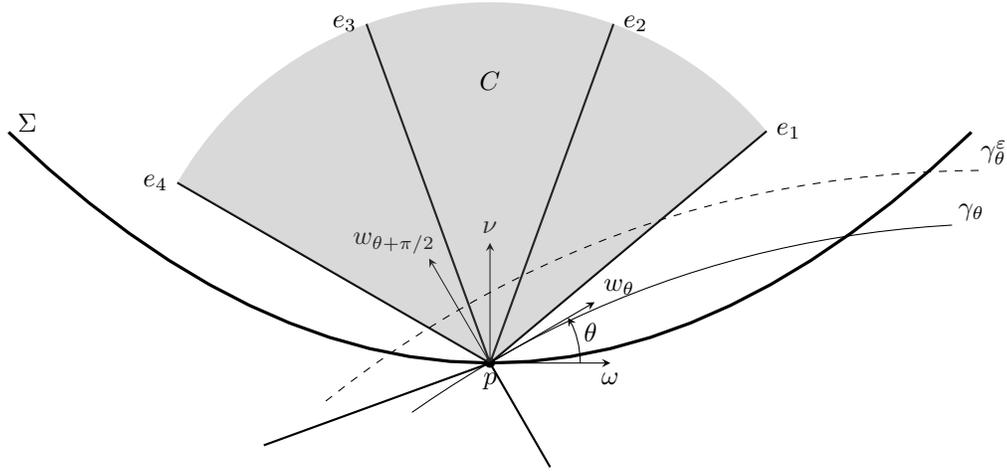
\begin{figure}[!h]

\begin{tikzpicture}[scale=0.8]

	\coordinate (P) at (0,0);
	\coordinate (A) at (-4,-8);
	\coordinate (B) at (9,2);
	
	\draw  [->,>= stealth] (P) -- ++(0,2);
	\draw  (0,2) node[above] {$\nu$};
	\draw  [->,>= stealth] (P) -- ++(2,0);
	\draw  (2,0) node[below] {$\omega$};
	\draw  [scale=1.98,rotate=20,->,>= stealth] (P) -- ++(1,0.18) node[above,right,shift={(0,0.2)}] {$w_\theta$};
	\draw  [scale=1.98,rotate=20,->,>= stealth] (P) -- ++(-0.18,1) node[above,left,shift={(0.2,0.2)}] {$w_{\theta+\pi/2}$};
	\draw [->,>= stealth] (1.5,0) arc (0:30.2:1.5);
	\draw (20:1.5) node[right]{$\theta$};
	\draw (0,0) node {$\bullet$};	
	\draw (0,0) node[below] {$p$};	
	
	\draw [very thick,domain=-8:8] plot (\x,{0.06*(\x)^2});
	\draw  (-8,4) node[right] {$\Sigma$};
	
	\draw [thick] (P) -- ++(40:6) node[right] {$e_1$};
	\draw [thick] (P) -- ++(70:6) node[right] {$e_2$};
	\draw [thick] (P) -- ++(110:6) node[left] {$e_3$};
	\draw [thick] (P) -- ++(150:6) node[left] {$e_4$};
	\draw [thick] (P) -- ++(200:4);
	\draw [thick] (P) -- ++(300:2);
	
	\fill[gray, opacity=0.3] (P) -- (40:6) arc (40:150:6) -- cycle;
	\draw (0,5) node[below] {$C$};
	
	\draw [domain=-1.5:8,rotate=20] plot (\x,{-0.03*(\x-3)^2+0.28});
	\draw  (8,2.8) node[below] {$\gamma_\theta$};
	\draw [dashed, domain=-2.6:9,rotate=20, shift={(-0.18,1)}] plot (\x,{-0.03*(\x-3)^2+0.28});
	\draw  (8,3.5) node[right] {$\gamma_\theta^\eps$};
\end{tikzpicture}
\caption{The corner and the variation defined in \cref{subsctn:corner def} and \cref{subsctn:corner variation}.}
\label{fig:corner reconstruction}
\end{figure}

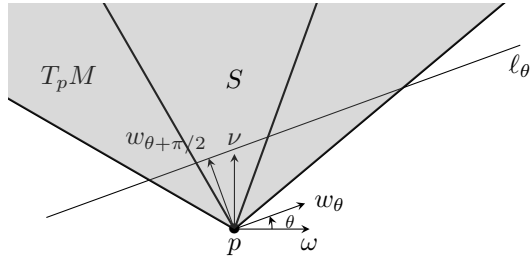
\begin{figure}[!h]

\begin{tikzpicture}
	\clip (-3,-1) rectangle (4,3);
	\coordinate (P) at (0,0);
	\draw (-2.2,2) node {$T_pM$};
	\draw [domain=-2.5:3.8] plot (\x,{1.064+0.364*\x}) node[below] {$\ell_\theta$};

	\draw  [->,>= stealth] (P) -- ++(0,1) node[above] {$\nu$};
	\draw  [->,>= stealth] (P) -- ++(1,0) node[below] {$\omega$};	
	\draw  [rotate=20,->,>= stealth] (P) -- ++(1,0) node[above,right] {$w_\theta$};
	\draw  [rotate=20,->,>= stealth] (P) -- ++(0,1) node[left, shift={(0.1,0.2)}] {$w_{\theta+\pi/2}$};
	\draw [->,>= stealth] (0.5,0) arc (0:20:0.5);
	\draw (10:0.6) node[scale=0.7,right]{$\theta$};
	\draw (0,0) node {$\bullet$};	
	\draw (0,0) node[below] {$p$};
	
	\draw [thick] (P) -- ++(40:5.355);
	\draw [thick] (P) -- ++(70:3.66);
	\draw [thick] (P) -- ++(120:3.99);
	\draw [thick] (P) -- ++(150:6.9);
	
	\fill[gray,dotted, opacity=0.3] (P) -- (4.12,3.46) -- (-5.99,3.46) -- cycle;
	\draw (0,2) node {$S$};
\end{tikzpicture}
\caption{Tangent data of \cref{subsctn:corner def} and \cref{subsctn:corner variation}.}
\label{fig:corner tangent data}
\end{figure}

\subsection{Reconstruction result}
\label{subsctn:corner result}
We may now state the main result of this section. 
\begin{lem}
	\label{lem:THE reconstruction in a corner}
	Take the notations and the framework introduced in \cref{subsctn:corner def} and \cref{subsctn:corner variation}. Suppose moreover that $J(0)=w_{\theta+\pi/2}$. Then we can reconstruct $f$ in the corner $C$ from the knowledge of the metric, the tiling and $I_C(\theta,\eps)$ for $(\theta,\eps)$ near $(0,0)$.
\end{lem}

This will be proven in \cref{subsctn:corner reconstruction}.

\begin{rk}
	If $f$ is supported in the corner $C$ then $I_C(\theta,\eps) = \ray(\gamma_\theta^\eps)$. Then the proof of the previous lemma and especially \eqref{eqn:values equal matrix derivatives} show that the reconstruction of $f$ uses the value of $\partial_\eps\ray(\gamma_\theta^\eps)_{|\eps=0}$ and its $N-1$ first partial derivatives in $\theta$ at $0$.
\end{rk}

\subsection{Conical functions in the Euclidean plane}\label{subsctn:conical functions}

The injectivity result \cite[Lemma~3.1]{ILS} can be adapted easily into a reconstruction result.

Let us denote the upper half plane by $H_+=\{(x,y)\in\R^2,y>0\}$. Let $\alpha_1>...>\alpha_N>\alpha_{N+1}$ and $a_1,...,a_N$ be any real numbers. Consider the conical function $f:H_+\to\R$ defined by 

\begin{equation}\label{eqn:def conical function}
    f(x,y)=
    \begin{cases}
        a_1, & \alpha_1y>x>\alpha_2y\\
        \vdots\\
        a_N, & \alpha_Ny>x>\alpha_{N+1}y\\
        0 & \text{ otherwise.}
    \end{cases}
\end{equation}
%
%
%
%
%
%
%

For $\theta$ near $0$, let $\ell_\theta$ be the line defined by $y=\frac{1}{\cos(\theta)}+\tan(\theta) x$ and denote $If(\theta)=\int_{\ell_\theta}{f\,ds}$. The line $\ell_\theta$ is exactly the same line $\ell_\theta$ defined previously in \cref{subsctn:corner variation}.

\begin{lem}\label{lem:conical functions}
    Let $f$ be defined as above. Then one can reconstruct the values $a_i$ given $If(\theta)$ for all $\theta$ near $0$. In fact, only the knowledge of the value and the $N-1$ first derivatives of $If$ at $0$ is sufficient.
\end{lem}

\begin{proof}
    The proof is almost the same as in \cite[Lemma~3.1]{ILS}. In our case, denoting also for $t$ near 0, $z_i^t=\frac{\alpha_i}{1-\alpha_it}$ we have for $\theta$ near 0 the following formula
    \begin{equation}
        If(\theta)=\frac{1}{\cos^2(\theta)}
        \sum_{i=1}^N a_i(z_i^{\tan(\theta)}-z_{i+1}^{\tan(\theta)}).
    \end{equation}
    Define then
    \begin{equation}\label{eqn:def of F}
        F(t)=\cos^2(\arctan(t))If(\arctan(t)),
    \end{equation}
    so that
    \begin{equation}
        F(t)=(z_1^t-z_2^t)a_1+...+(z_{N}^t-z_{N+1}^t)a_N.
    \end{equation}
    Taking derivatives with respect to $t$ at $0$, we may then write a similar equation of \cite[Eq.~(8)]{ILS} but with a general right hand side
    \begin{equation}\label{eqn:euclidean case with rhs}
        A
        \left ( \begin{array}{c}
             a_1 \\
             \vdots \\
             a_N
        \end{array}\right)
        = b
    \end{equation}
    where $b$ is a vector whose coordinates are $F$ and its derivatives with respect to $t$ at $0$
    \begin{equation}
        b= 
        \left ( \begin{array}{c}
             F(0) \\
             \vdots \\
             F^{(N-1)}(0)
        \end{array}\right)
    \end{equation}
    and $A$ is given by \cite[Eq.~(9)]{ILS} that we recall here
    \begin{equation}\label{eqn:def of A}
        A=
        \left (
        \begin{array}{ccc}
             \alpha_1-\alpha_2 & ... & \alpha_N - \alpha_{N+1}  \\
             \vdots & & \vdots \\
             \alpha_1^N-\alpha_2^N & ... & \alpha_N^N -\alpha_{N+1}^N
        \end{array}
        \right ).
    \end{equation}
    
    The proof of \cite[Lemma~3.1]{ILS} that $A$ is invertible actually gives us also that its inverse can be explicitly determined from its coefficients (the inversion of $A$ is only based on the inversion of a Vandermonde matrix which can be computed explicitly, see \cite{KNU97} for instance). The computation of $A^{-1}$ and \eqref{eqn:euclidean case with rhs} allow us then to conclude.
\end{proof}


\subsection{Reduction to the Euclidean case}

The reduction to an Euclidean problem \cite[Lemma~4.2]{ILS} is true for any variation through geodesics which moves infinitesimally in the right direction at the corner.

\begin{lem}\label{lem:corner reconstruction}
    Suppose that $f : \tild{M} \to \R$ is a piecewise constant function on the tiling $T$. For all $\theta$ in some neighborhood of $0$,
    if $J(0)=w_{\theta+\pi/2}$, then 
    \begin{equation*}
        \lim_{\eps \to 0}\frac{1}{\eps}\int_{\gamma_\theta^\eps\cap\,C}{f\,ds}=\int_{\ell_\theta\cap\,S}{T_pf\,ds}.
    \end{equation*}
\end{lem}

\begin{proof}
	
    The proof of this fact is exactly the same as the proof of \cite[Lemma~4.1]{ILS} and \cite[Lemma~4.2]{ILS}.
    We replace here the geodesics $\gamma^h_v$ considered in \cite{ILS} by the geodesics $\gamma_\theta^\eps$ of our more general variation through geodesics. The implicit function theorem still holds because it uses only data on $\gamma_\theta$ which is the same in \cite{ILS} and here. Computations for the derivatives of the meeting times can be done also the same way since they rely on the exact same information at the corner at $p$ (in particular $\partial_\eps\gamma_\theta^\eps|_{\eps=0}=J$ is assumed to be equal to $w_{\theta+\pi/2}$ at $t=0$). 
\end{proof}

\subsection{Proof of the reconstruction in a corner}
\label{subsctn:corner reconstruction}
\begin{proof}[Proof of \cref{lem:THE reconstruction in a corner}]
	Enumerate from $1$ to $N+1$ by increasing angles the edges of the tiles of $\mathcal{C}$ and denote $\Delta_i \in \mathcal{C}$ the tile having edges $e_i$ and $e_{i+1}$ at $p$. Denote also $a_i$ the values of $f$ on $\Delta_i$. The tangent function $T_pf$ is in this case a conical function defined on conical sectors as in \eqref{eqn:def conical function} with here $\alpha_i=\frac{1}{\tan(\theta_i)}$ where $\theta_i$ is the angle of the edge $e_i$ in the corner.
	
	By \cref{lem:corner reconstruction} and \cref{lem:conical functions} we can then reconstruct $f$ in the corner $C$ by the formula
    \begin{equation} \label{eqn:values equal matrix derivatives}
            \left ( \begin{array}{c}
             a_1 \\
             \vdots \\
             a_{N}
            \end{array}\right)
            =A^{-1}
            \left ( \begin{array}{c}
             F(0) \\
             \vdots \\
             F^{(N-1)}(0)
            \end{array}\right)
        \end{equation}
        where $A$ defined as in \eqref{eqn:def of A} and
        \begin{equation}\label{eq:def of F}
        	F(t)=\cos^2(\arctan(t))\partial_\eps I_C(\arctan(t),0).
        \end{equation} 

\end{proof}

    \section{Near the boundary}\label{sctn:boundary}
        In this section, we recover the values of $f$ near the boundary by asymptotic computations at vertices of the tiling. We start by reconstructing $f$ on tangent tiles and then we deal with corner tiles touching the boundary using \cref{sctn:corner}.

Let $p\in\bdry$ be a vertex of the tiling. At $p$, there are two tangent tiles $\Delta_1$ and $\Delta_N$ which must have an edge all along the boundary and a certain number of corner tiles (maybe none) $\Delta_2,...,\Delta_{N-1}$ which intersects the boundary only at $p$. Enumerate the tiles by increasing maximum angles of the tangent cones.

Denote by $\nu$ the inward pointing normal to $\bdry$ at $p$ and take $\omega\in T_pM$ so that $(\omega,\nu)$ is an orthonormal basis of $T_pM$. Extend smoothly the metric outside of $M$ and denote then as in \cref{subsctn:corner variation} $\gamma_\theta$ the geodesic starting from $p$ with initial velocity $w_\theta$.

\subsection{Tangent tiles}
Denote $a_1$ and $a_N$ the values of $f$ on $\Delta_1$ and $\Delta_N$ respectively (see figure \ref{fig:side tiles}). We will show the following reconstruction lemma.

\begin{lem}\label{lem:side tiles reconstruction}
    For $i=1,N$,
    \begin{equation*}
		a_i=\frac{\kappa}{2}\,\lim_{\eps\to 0^+}\frac{1}{\eps}\ray(\gamma_i^\eps),
    \end{equation*}
    where $\kappa$ is the curvature of $\bdry$ at $p$.
\end{lem}

The geodesics used in this lemma are defined as follow : $\gamma_1^\eps$ (resp. $\gamma_N^\eps$) is the geodesic starting from $p$ with initial velocity $ \omega+\eps\nu$ (resp. $ - \omega+\eps\nu$).

\begin{figure}[h]

	\begin{tikzpicture}[scale=1]
		
		\draw (0,0) node {$\bullet$};
		\draw (0,0) node[below] {$p$};
		
		\draw [very thick, domain=-4:7] plot (\x,{0.05*(\x)^2});
		\draw  (-3,1) node {$\bdry$};
		
		\draw [domain=-3.5:4] plot (\x,{-0.04*(\x)^2}) node[right] {$\gamma$};
	
		\draw [dashed, domain=-2:6.6,rotate=30] plot (\x,{-0.04*(\x)^2}) node[below] {$\gamma_1^\eps$};
		\draw [thick] (0,0) -- ++(50:4) ;
		\draw (40:4.2) node{$\Delta_1$};
		\draw [thick] (0,0) -- ++(130:4);
		\draw (140:3) node {$\Delta_N$};
		
		\draw [->,>= stealth] (0.7,0) arc (0:130:0.7);
		\draw (140:1) node[above right]{$\theta_N$};
		\draw [->,>= stealth] (1,0) arc (0:50:1);
		\draw (30:1) node[right]{$\theta_1$};
		
		\draw[->,>= stealth] (0,0) -- (0:2) node[below] {$\omega$};
		\draw[->,>= stealth] (0,0) -- (90:2) node[above,left] {$\nu$};
		\draw[->,>= stealth] (0,0) -- (30:2.31) node[above right] {$\dot\gamma_1^\eps(0)$};
		
	\end{tikzpicture}
	\caption{Tangent tiles setting for the reconstruction}
	\label{fig:side tiles}
	\end{figure}
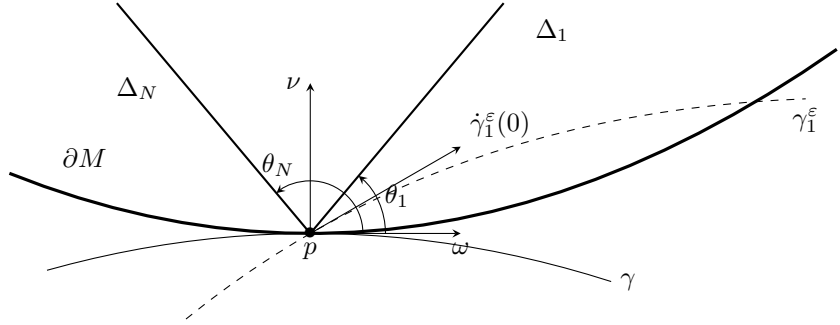

\Cref{lem:side tiles reconstruction} is a direct corollary of the following lemma since $\ray(\gamma_i^\eps)=a_it_i(\eps)$ for $\eps$ small enough.

\begin{lem}
	\label{lem:time side tiles reconstruction}
	The non trivial time $t_i(\eps)$ at which $\gamma_i^\eps$ touches the boundary is smooth in $\eps$ and
	\begin{equation*}
		t_i'(0)=\frac{2}{\kappa},
	\end{equation*}
	where $\kappa$ is the curvature of $\bdry$ at $p$.
\end{lem}

\begin{proof}
	The smoothness of $t_i$ is a known fact (see \cite{sharafutdinov1999ray,U13}). We recall here the proof of \cite[Lemma~2.5]{U13} to show the formula for the derivative. Let $\rho$ be a boundary defining function for $\bdry$ around $p$. Define then the following smooth function in a neighborhood of $(0,0)$
	\begin{equation}\label{eqn:def of h}
		h(\eps,t)=\rho(\gamma_i^\eps(t))
	\end{equation}
	so that $\gamma_i^\eps(t)\in\bdry \Leftrightarrow h(\eps,t)=0$ around $(0,0)$.
	
	We have that for all $\eps$ near $0$
	\begin{equation}
		\partial_t h(\eps,0) = \inner{\nu}{\dot\gamma_i^\eps(0)}
	\end{equation}
	and
	\begin{equation}
		\partial^2_t h(\eps,0) =: c(\eps).
	\end{equation}
	So the Taylor's expansion in $t$ for $h$ at $0$ gives a smooth function $R$ such that for all $(\eps,t)$ near $(0,0)$ 
	\begin{equation}
		h(\eps,t)=\inner{\nu}{\dot\gamma_i^\eps(0)} \,t + \frac{c(\eps)}{2}\,t^2 + R(\eps,t)\,t^3.
	\end{equation}
	and $R(\eps,t)\to 0$ when $t\to 0$.
	
	Therefore, if we define near $(0,0)$
	\begin{equation}
		g(\eps,t)=\inner{\nu}{\dot\gamma_i^\eps(0)} + \frac{c(\eps)}{2}\,t + R(\eps,t)\,t^2,
	\end{equation}	
	the non trivial ending time $t(\eps)$ satisfies
	\begin{equation}
		\label{eqn:g ending times}
		g(\eps,t(\eps))=0.
	\end{equation}
	Since $\partial_t g(0,0)=c(0)/2=-\kappa/2<0$, the implicit function theorem gives that the non trivial ending time $t(\eps)$ is smooth and that differentiating \eqref{eqn:g ending times} we can conclude using that
	\begin{equation}
		\partial_\eps g(0,0) = \inner{D_tJ(0)}{\nu} = 1.
	\end{equation}
\end{proof}
%

\begin{rk}\label{rk:side tiles noncanonical method}
    We could have reconstructed the values $a_i$ by considering a geodesic $\gamma_i^\eps$ for a fixed $\eps$ small enough so that it is entirely contained in $\Delta_i$ and use the formula $a_i=\frac{1}{len(\gamma_i^\eps)}\int_{\gamma_i^\eps}{f\,ds}$. Even if this method does not involve any derivatives, it gives a result which implies geodesics of the variation instead of just the Jacobi field. \Cref{lem:side tiles reconstruction} expresses the values of $f$ on tangent tiles around $p$ only with infinitesimal data at $p$.
\end{rk}

\subsection{Corner tiles}\label{subsctn:middle tiles}

We are now left with corner tiles at the boundary that we will reconstruct using \cref{sctn:corner} : reducing the Riemannian problem to a Euclidean one. To do so we only need to compute $\int_{\gamma^\eps_\theta\cap C}{f\,ds}$ for well chosen geodesics $\gamma_\theta^\eps$ where the corner $C$ is defined as in \cref{sctn:corner}. 

The variations through geodesics we use are the same as in \cite[Section~4]{ILS}. For $\theta$ near $0$, and $\eps\geq 0$ small enough let $w_\theta(\eps)$ be the unit vector defined as the parallel transport of $w_\theta$ along the geodesic through $w_{\theta+\pi/2}$ by distance $\eps$. Define now the geodesic $\gamma_{\theta}^\eps$ as the maximal geodesic through $w_\theta(\eps)$. Denote $J_\theta$ the Jacobi field of this variation through geodesics in $\eps$.

    



Denote $t_1^\theta(\eps)$ and $t_{N}^\theta(\eps)$ the respective positive and negative ending times of $\gamma_\theta^\eps$. Denote also $p_i^\theta:=\gamma_\theta(t_i^\theta(0))$.

\begin{lem}
	\label{lem:ending times theta non null}
	With the previous notations, for $\theta$ sufficiently close to $0$ and $\theta\ne0$, then both ending times $t_i^\theta$ are smooth for $\eps$ near $0$ and
    	\begin{equation*}
			t_i^\theta(\eps) = t_i^\theta(0) - \frac{\inner{J_\theta(t_i^\theta(0))}{\nu(p_i^\theta)}}{\inner{\dot\gamma_\theta(t_i^\theta(0))}{\nu(p_i^\theta)}} \,\eps + \underset{\eps\to 0}{o}(\eps).
		\end{equation*}
\end{lem}

\begin{proof}
	Define $h$ as in the proof of \cref{lem:time side tiles reconstruction} but with $\rho$ a defining function for the boundary around the endpoint $p_i^\theta$, so that $\gamma_\theta^\eps(t)\in\bdry \Leftrightarrow h(\eps,t)=0$ around $(0,t_i^\theta(0))$.
	
	Since $\theta\ne 0$, the strict convexity of the boundary gives that $\gamma_\theta^\eps$ does not touch $\bdry$ tangentially at its endpoints. We thus have 
	\begin{equation}
		\partial_t h(0,t_i^\theta(0)) = \inner{\nu(p_i^\theta)}{\dot\gamma_\theta(t_i^\theta(0))}\ne 0.
	\end{equation}
	Therefore, the implicit function theorem gives us the smoothness of $t_\theta$ and we can conclude differentiating $h(\eps,t_i^\theta(\eps))=0$ and using that $\partial_\eps h(0,t_i^\theta(0))=\inner{J_\theta(t_i^\theta(0))}{\nu(p_i^\theta)}$.
\end{proof}

The only case where a non smoothness appears is when $\theta=0$. The ending time is equivalent to a square root of $\eps$ near $0$ because of the strict convexity of the boundary at $p$. The next term in the asymptotical expansion comes from the jerk of $\bdry$ at $p$ or in other words the variation of the curvature at $p$.

\begin{lem}
	\label{lem:ending times theta null}
    With the previous notations, if $\theta=0$, then
		\begin{align*}
		  t_1(\eps) = \sqrt{\frac{2}{\kappa}\eps\,} - \frac{j}{3\kappa^2}\,\eps 
		  + \underset{\eps\to 0}{o}(\eps) \\
		  t_N(\eps) = - \sqrt{\frac{2}{\kappa}\eps\,} - \frac{j}{3\kappa^2}\,\eps 
		  + \underset{\eps\to 0}{o}(\eps)
		\end{align*}
	  where $\kappa>0$ is the curvature and $j$ the jerk of $\bdry$ at $p$ for $\dot\gamma(0)$.
\end{lem}

\begin{figure}[h]

	\begin{tikzpicture}[scale=0.8]
	
		\draw (0,0) node {$\bullet$};
		\draw (0,0) node[below] {$p$};
		
		\draw [domain=-6:6] plot (\x,{0.04*(\x)^2});
		\draw  (-5.3,1.5) node {$\bdry$};
		
		\draw [domain=-4:4] plot (\x,{-0.04*(\x)^2}) node[right] {$\gamma$};
	
		\draw [dashed, domain=-5.5:5.5] plot (\x,{-0.04*(\x)^2+1}) node[right] {$\gamma^\eps$};
		
		\draw[->,>= stealth] (0,0) -- (90:2) node[above,right] {$\nu$};
		
	\end{tikzpicture}
	\caption{The variation $\gamma^\eps$ when $\theta=0$.}
	\label{fig:bdry ending times}
\end{figure}
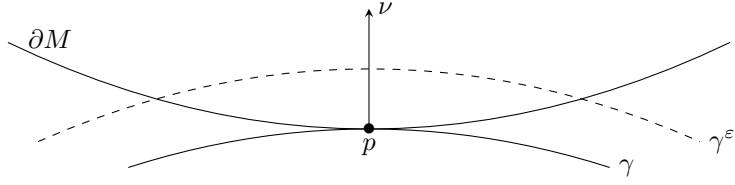

\begin{rk}\label{rk:sqrt singularity at boundary}
	Although it is known (see \cite[Lemma~3.2.1]{sharafutdinov1999ray} for instance) that the ending time is a smooth function of the starting point and velocity on the inward pointing boundary unit sphere bundle, here geodesics have a starting point in $\inter{M}$ and an initial velocity tangential to the strictly convex boundary which causes the square root singularity of \cref{lem:ending times theta null}
\end{rk}

\begin{proof}
	The idea of the proof is to use the Taylor's expansions for a function parametrizing the boundary near $p$ given by \cref{lem:parametrization boundary} and for the variation $\Gamma(\eps,t)=\gamma^\eps(t)$ near $(0,0)$.
		
	In the normal coordinates at $p$ given by the orthonormal basis $(\dot{\gamma}(0),\nu)$ of $T_pM$ there exist a smooth function $h$ defined near $0$ in $\R$ such that
	\begin{equation}\label{eqn:taylor of the bdry}
		h(x)=\frac{\kappa}{2} x^2 + \frac{j}{6} x^3 + \underset{x\to 0}{o}(x^3),
	\end{equation}
	and for $z=(x,y)$ in our local coordinates
	\begin{equation}\label{eqn: h is parametrizing}
		z=(x,y)\in\bdry \Leftrightarrow y = h(x)
	\end{equation}
	by \cref{lem:parametrization boundary}.
	
	We may then write the Taylor's expansion of $\Gamma$ near $(0,0)$ in our local coordinates
	\begin{equation}
		\Gamma(\eps,t)
		=\eps \nu
		+ t \omega
		+ \mathcal{O}(|\eps|^2+|t|^2),
	\end{equation}
	where we do not write the terms useless in the asymptotic expansion at order $\sqrt \eps$. Denote for $(\eps,t)$ close to $(0,0)$, $\Gamma(\eps,t)=\begin{pmatrix}
			x_\eps(t)\\
			y_\eps(t)
		\end{pmatrix}$. If $\Gamma(\eps,t)\in\bdry$ for $(\eps,t)$ near $(0,0)$ then by \eqref{eqn:taylor of the bdry} and \eqref{eqn: h is parametrizing},
	\begin{equation} \label{eqn:bdry rough expansion} 
		y_\eps(t)=\frac{\kappa}{2} x_\eps^2(t) + \frac{j}{6} x_\eps^3(t) + o(|x_\eps(t)|^3).
	\end{equation}
	Using \eqref{eqn:bdry rough expansion} one has for $(\eps,t)$ near $(0,0)$ such that $\Gamma(\eps,t)\in\bdry$ the following expansion
	\begin{equation}\label{eqn:taylor final}
		\eps - t^2\frac{\kappa}{2} = \mathcal{O}(|\eps|^2+|t|^2).
	\end{equation}
	By \cite[Lemma~2.3]{U13}, the times $t_i(\eps)$ are continuous. They also satisfy $t_i(0)=0$ and then \eqref{eqn:taylor final} when $\eps$ is near $0$ by definition. Moreover $\kappa>0$ by strict convexity of $\bdry$ so we have
	\begin{align}
		t_1(\eps) \underset{\eps\to 0^+}{\sim} \sqrt{\frac{2}{\kappa}\eps\,},\\
		t_N(\eps) \underset{\eps\to 0^+}{\sim} -\sqrt{\frac{2}{\kappa}\eps\,}.
	\end{align}
	
	Then we define $s_1(\eps)=t_1(\eps)-\sqrt{\frac{2}{\kappa}\eps\,}$ and $s_N(\eps)=t_N(\eps)+\sqrt{\frac{2}{\kappa}\eps\,}$ so that in both case $s_i(\eps)= o(\sqrt{\eps})$. To get an equivalent of $s_i(\eps)$ in 0, we use a higher order Taylor's expansion of $\Gamma$ at $(0,0)$ writing only the terms useful in an asymptotic expansion of \eqref{eqn:bdry rough expansion} at order $\eps^{3/2}$ :
	\begin{equation} \label{eqn:first high taylor}
		\Gamma(\eps,t)
		=\eps \nu
		+ t \omega
		+ \frac{t^2}{2}\partial^2_{tt}\Gamma(0,0)
		+ \frac{t^3}{6}\partial^3_{ttt}\Gamma(0,0)
		+ o(\eps^{3/2}).
	\end{equation}
	In normal coordinates, $\partial^2_{tt}\Gamma(0,0)=0$ and $\partial^3_{ttt}\Gamma(0,0)=0$.
	Then if $\Gamma(\eps,t)\in\bdry$ for $(\eps,t)$ near $(0,0)$, by \eqref{eqn:first high taylor} and \eqref{eqn: h is parametrizing},
	\begin{equation}
		\eps = \frac{\kappa}{2} t^2 + \frac{j}{6} t^3 + \underset{\eps\to 0}{o}(\eps^{3/2})
	\end{equation}
	Simplifying this equation we get
	\begin{equation}
		s_i(\eps)+o(s_i(\eps)) = - \frac{j}{3\kappa^2}\eps + o(\eps),
	\end{equation}
	which ends the proof.
\end{proof}

Now we may state the lemma expressing the integrals of $f$ in the corner. Let $\sigma_1$ and $\sigma_N$ be the two unit speed curves starting from $p$ parametrizing the edges of $\Delta_1$ and $\Delta_N$ respectively which are not along $\bdry$. Define then $\theta_i\in(-\pi,\pi]$ such that $\dot\sigma_i(0)=w_{\theta_i}$.

\begin{lem}\label{lem:middle tiles}
    With the previous notations, for $\theta$ sufficiently close to $0$,
    \begin{enumerate}
    	\item If $\theta\ne0$, then 
    	\begin{multline*}
		\int_{\gamma^\eps_\theta\cap C}{f\,ds}=\ray(\gamma^\eps_\theta) + a_Nt_{N}^\theta(0) - a_1 t_1^\theta(0) + \\ \left(
			\frac{\inner{J_\theta(t_1^\theta(0))}{\nu}}{\inner{\dot\gamma_\theta(t_1^\theta(0))}{\nu}} 
			- \frac{\inner{J_\theta(t_N^\theta(0))}{\nu}}{\inner{\dot\gamma_\theta(t_N^\theta(0))}{\nu}} 
			+ a_1\beta_1 - a_N\beta_N \right) \,\eps + \underset{\eps\to 0}{o}(\eps).
	\end{multline*}
    	\item If $\theta=0$, then 
    	\begin{equation*}
        \int_{\gamma^\eps\cap C}{f\,ds}=\ray(\gamma^\eps)-\sqrt{\frac{2\,\eps}{\kappa}}(a_1+a_N) + \left[ a_1\beta_1-a_N\beta_N+ \frac{j}{3\kappa^2}(a_1-a_N) \right]\,\eps + \underset{\eps\to 0}{o}(\eps).
    \end{equation*}
    \end{enumerate}
    where $\beta_i=\frac{\cos(\theta_i-\theta)}{\sin(\theta_i-\theta)}$, $\kappa>0$ the curvature and $j$ the jerk of $\bdry$ at $p$ for $\dot\gamma(0)$.
\end{lem}

\begin{proof}
    The proof of \cite[Lemma~4.1]{ILS} gives two smooth times $s_1$ and $s_N$ parametrizing the intersection of $\gamma^\eps_\theta$ with $\sigma_1$ and $\sigma_N$ respectively. It gives also that these times satisfy
    \begin{align}
        & s_1'(0) = -\frac{\sin(\angle(\dot{\sigma_1}(0),w_{\theta+\pi/2}))}{\sin(\angle(\dot{\sigma_1}(0),w_\theta))}=\frac{\cos(\theta_1-\theta)}{\sin(\theta_1-\theta)},\\
        & s_N'(0) = -\frac{\sin(\angle(\dot{\sigma_N}(0),w_{\theta+\pi/2}))}{\sin(\angle(\dot{\sigma_N}(0),w_\theta))}=\frac{\cos(\theta_N-\theta)}{\sin(\theta_N-\theta)}.
    \end{align}
    
    We can apply \cref{lem:ending times theta non null} to obtain a formula of the ending times if $\theta\ne0$ and \cref{lem:ending times theta null} if $\theta=0$.
    
    Since we can write
    \begin{equation}
        \ray(\gamma^\eps_\theta)=\int_{\gamma^\eps_\theta\cap C}{f\,ds}+ a_1(t_1(\eps)-s_1(\eps)) + a_N(s_{N}(\eps)-t_{N}(\eps)),
    \end{equation}
    the conclusion of the lemma follows.
\end{proof}

The previous lemma allows us to recover $f$ near $p$ using the \cref{sctn:corner} of reconstruction at a corner. We therefore recovered $f$ in a neighborhood of the boundary  using the metric, the tiling, the ray transform and only the Jacobi fields of the variations used for the reconstruction.

     \section{Inside the manifold}\label{sctn:iteration}
         We wish now to reconstruct $f$ inside the manifold. The idea of the proof is the same as in \cite[Theorem~5.3]{ILS}: we use the strictly convex foliation to reconstruct $f$ in all $M$ by starting the reconstruction from the boundary and then iterate the argument at level sets of $\varphi$ using the values of $f$ already recovered. We now prove a local reconstruction lemma around a point of a level set of the foliation.

\begin{lem}\label{lem:reconstruction noass iteration step}
    Let $M$ be a compact nontrapping Riemannian two-dimensional manifold with strictly convex boundary. Let $f$ be a piecewise constant function on a regular tiling. Let $\varphi$ be a strictly convex foliation of $M$ and $\min\varphi< c \leq\max\varphi$. Denote $\Sigma=\{\varphi=c\}$ and let $p\in\Sigma\cap\inter{M}$. If the values of $f$ are known in $\{\varphi>c\}$, then one can reconstruct the values of $f$ near $p$ from the knowledge of $\ray$, the metric and the tiling.
\end{lem}

\begin{rk}\label{rk:reconstruction noass}
	The proof of the previous lemma (see \cref{subsctn:proof reconstruction noass}) uses at $p$ a 2-parameter family of geodesics $\gamma_\theta^\eps$ such that for each fixed $\theta$, $\gamma_\theta^\eps$ is a smooth variation through geodesics in $\eps$ starting from the almost tangent geodesic $\gamma_\theta$ defined in \cref{subsctn:corner variation}.
\end{rk}

We first prove \cref{main:noass} thanks to this local argument.

\begin{proof}[Proof of \cref{main:noass}]
    There are only a finite number of simplices $\Delta_1,...,\Delta_m$ in the tiling. Denote by $c_1>...>c_K$ the distinct elements of the set $\{\max_{\Delta_i}\varphi,\,1\leq i\leq m\}$.
    
    Since $c_1=\max\varphi$,  by \cite[Lemma~2.5]{PSUZ16} we have that $\{\varphi= c_1\}\subset \bdry$. Therefore by the \cref{sctn:boundary} we can reconstruct $f$ on the simplices $\Delta_i$ such that $\max_{\Delta_i}\varphi=c_1$ and hence in the set $\{\varphi>c_2\}$. The values being given by the work of \cref{sctn:boundary}.
    
    Now we iterate the argument. We wish to reconstruct $f$ on each tile $\Delta_i$ such that $\max_{\Delta_i}\varphi=c_2$.
    To do so, let $p\in \{\varphi=c_2\} \cap\Delta_i$ with $\max_{\Delta_i}\phi = c_2$. If $p\in\bdry$ we can reconstruct $f$ near $p$ by \cref{sctn:boundary} and if $p$ is in $\text{int}(M)$ we can reconstruct $f$ near $p$ by \cref{lem:reconstruction noass iteration step}.
    
    Iterating this method, we reach the index $K$ and we reconstruct $f$ everywhere on $M$.
\end{proof}

\subsection{Proof of \cref{lem:reconstruction noass iteration step}} \label{subsctn:proof reconstruction noass}
Without any further assumptions on the manifold and the tiling, one can recover $f$ from the full knowledge of the tiling and the metric without almost any changes from the injectivity method of \cite{ILS}.

\begin{proof}
    Since the full tiling and the metric are known, one can compute the integrals of $f$ along a geodesic in any tile on which we know its value.
    
    To prove this local reconstruction lemma, we follow the proof of \cite[Lemma~5.1]{ILS}. Denote then $\nu$ the normal to $\Sigma$ at $p$ pointing toward $\{\phi<c\}$ and recall the three mutually exclusive types of tiles containing $p$ of \cref{subsctn:foliation}. Denote here again $H_\pm=\{u\in T_pM,\,\pm \,\inner{\nu}{u} >0\}$ and $H_0=T_p\Sigma$.
    
    If $\Delta$ is of the first type it intersects $\{\phi>c\}$ so $f$ is already recovered on $\Delta$.
    
    If $\Delta$ is a tangent tile, take $w\in H_0\cap C_p\Delta$ and define as in \cite[Lemma~5.1]{ILS} the family of geodesics $\gamma^\eps$ starting at $p$ with initial velocity $\dot\gamma^\eps(0)=w+\eps\nu$. 
    
    Since $\Sigma$ is strictly convex, for $\eps>0$ small enough, $\gamma^\eps$ touches back $\Sigma$ near $p$. Before the starting point and after touching $\Sigma$, $\gamma^\eps$ stays in $\{\phi>c\}$ by \cref{lem:geodesics in foliated manifolds} so the values of the tiles it goes through in this set are known. 
    
    It is also shown in \cite[Lemma~5.1]{ILS} that for $\eps>0$ small enough, before touching $\Sigma$, $\gamma^\eps$ must stay either
    \begin{itemize}
        \item only in $\Delta$ or
        \item in $\Delta\cup\tilde{\Delta}$ where $\tilde{\Delta}$ is the unique simplex such that $\Delta$ and $\tilde{\Delta}$ have a common edge tangent to $\Sigma$ at $p$. Moreover, the value of $f$ on $\tilde{\Delta}$ is known because $\tilde{\Delta}$ is of the first type.
    \end{itemize}

    In both cases we know the values of $f$ on $\Delta^C$ and
    we can compute $\int_{\gamma^\eps\cap\Delta}{f\,ds}=\int_{\gamma^\eps}{f\,ds}-\int_{\gamma^\eps\cap \Delta^C}{f\,ds}$. We therefore find the value of $f$ on $\Delta$ since it is constant on this tile and since the knowledge of the tiling and the metric allow us to compute the length of $\gamma^\eps$ in $\Delta$.

    We are left with the corner tiles which we recover thanks to \cref{sctn:corner}. If there are corner tiles, $p$ is a vertex. The geodesics $\gamma_\theta^\eps$ we use are the one constructed in \cref{subsctn:middle tiles} replacing $\bdry$ by $\Sigma$. They stay in $\{\phi>c\}$ outside the tiles having $p$ as a vertex by \cref{lem:geodesics in foliated manifolds} for small enough parameters. We can therefore compute the integrals $\int_{\gamma^\eps_\theta\cap C}{f\,ds}$ for $\eps$ and $\theta$ near $0$ and conclude by \cref{lem:THE reconstruction in a corner}.
    
\end{proof}

      \section{Simple manifolds and geodesic tilings}\label{sctn:simplegeod}
      	As we said in the introduction, we will now improve the local reconstruction result \cref{lem:reconstruction noass iteration step} in simple manifolds with geodesic tilings. The improvement is summed up in the following lemma.

\begin{lem}\label{lem:reconstruction simplegeod iteration step}
    Take the same notations and assumptions as in \cref{lem:reconstruction noass iteration step}. Assume in addition that $M$ is simple and that the tiling is geodesic. If the values of $f$ are known in $\{\varphi>c\}$, then one can reconstruct $f$ near $p$ from the knowledge of $\ray$, the metric and the tangent vectors to the edges of the tiling at their intersection point with the geodesics $\gamma_\theta$ defined at each vertex (see \cref{subsctn:corner variation}) for small enough $\theta$.
\end{lem}

\begin{rk}\label{rk:reconstruction simplegeod}
	Constrasting with \cref{rk:reconstruction noass}, the proof of the previous lemma (see \cref{subsctn:proof reconstruction simplegeod}) uses at $p$ the 1-parameter family of geodesics $\gamma_\theta$ and only a non-vanishing Jacobi field $J_\theta$ along $\gamma_\theta$ instead of the full variation through geodesics in $\eps$.
\end{rk}

The key lemma to prove the previous result is a formula which expresses the integrals of $f$ over a variation through geodesics in terms of the meeting times of these geodesics with the edges of the tiling. This is only possible under some assumptions on the tiling and the variation. The reconstruction will then use only the initial derivatives of these meeting times which can be expressed thanks to the Jacobi field of the variation. We will finally show that the geometric conditions on the manifold and the tiling imply that the assumptions are satisfied.

\subsection{Parametrization of the integrals over a variation through geodesics}

\subsubsection{Result}

Let $p\in \inter{M}$ and $w\in T_pM$ be any unit tangent vector at $p$.
Denote $\gamma$ the maximal geodesic starting from $p$ with velocity $w$.

Extend smoothly the metric outside of $M$ around the meeting points of $\gamma$ with $\partial M$ so that we can define $\Gamma : (-\eps_0,\eps_0)\times[a,b] \to M$ a smooth normal variation through geodesics touching the boundary starting from $\gamma$. Denote $\gamma^\eps = \Gamma(\eps,\cdot)$ the geodesic parametrized by $\eps$. Denote $J$ the Jacobi field associated to $\Gamma$ and denote 
$\tubb[\eps_0] =\Gamma((0,\eps_0)\times[a,b])$ and 
$\tub[\eps_0] =\Gamma([0,\eps_0)\times[a,b])$.

\begin{lem}\label{lem:integral if nontangential nonvanishing}
    If 
    \begin{enumerate}[label=(\roman*)]
        \item \label{ass:nontangential} every interior edge is either entirely along $\gamma$ or never tangential to it and
        \item \label{ass:nevervanishing} $\forall t\in[a,b],\,J(t)\ne 0$,
    \end{enumerate}
    then there exist $\eps_1>0$, $N\in\N$, ordered smooth functions $t_i : [0,\eps_1)\to[a,b]$ for $i\in\libr1,N+1\ribr$ and $(a_i)_{i\in\libr1,N\ribr}$ values of $f$ associated to tiles $\Delta_i$ such that for all $\eps\in(0,\eps_1)$,
    \begin{equation*}
		f\circ\gamma^\eps\equiv a_i \text{ on } (t_{i}(\eps),t_{i+1}(\eps)).
	\end{equation*}
    so in particular,
    \begin{equation*}
        \forall \eps\in(0,\eps_1), \intef{\eps}=\sum_{i=1}^{N}{a_i(t_{i+1}(\eps)-t_{i}(\eps))}.
    \end{equation*}
\end{lem}

We next state this lemma in the simple case where the variation through geodesics is a diffeomorphism and prove that this restricted result allows us to prove the previous one. The next lemma is proved in the \cref{subsubsctn:Proof of nontnonvsimpl}.

\begin{lem}\label{lem:integral if nontangential nonvanishing and simple}
    Take the same assumptions as in \cref{lem:integral if nontangential nonvanishing} and suppose moreover that $\gamma$ is simple. Then the conclusion of \cref{lem:integral if nontangential nonvanishing} holds.
\end{lem}


\begin{proof}[Proof of \cref{lem:integral if nontangential nonvanishing}]
    First, $\gamma$ has no loops because $M$ is a nontrapping manifold. Therefore, by \cite[Lemma~7.2]{KS} one can deduce that there are only finitely many times $a\leq \tau_1<...<\tau_m\leq b$ such that $\gamma$ intersects itself at $\gamma(\tau_j)$. Taking for every $j\in\libr1,m-1\ribr,\,s_j\in(\tau_j,\tau_{j+1})$ one has that $\gamma$ is simple on $[s_j,s_{j+1}]$. Therefore by \cref{lem:integral if nontangential nonvanishing and simple}, denoting the variation through geodesics restricted in time $\Gamma_j:(-\eps_0,\eps_0)\times[s_j,s_{j+1}]\to M$ and $\gamma_j^\eps = \Gamma_j(\eps,\cdot)$, we have that there exist $\eps_1>0$, $N\in\N$ and ordered smooth functions $t_{i,j}:[0,\eps_1)\to[s_j,s_{j+1}]$ and constants $a_{i,j}\in\R$ such that for all $\eps\in(0,\eps_1)$,
    \begin{equation}
		f\circ\gamma^\eps\equiv a_{i,j} \text{ on } (t_{i,j}(\eps),t_{i+1,j}(\eps))
    \end{equation}
    and the conclusion follows since the $s_j$ are ordered.
\end{proof}

\subsubsection{Proof of the \cref{lem:integral if nontangential nonvanishing and simple}}
\label{subsubsctn:Proof of nontnonvsimpl}

The global idea is to construct the functions $t_i$ such that $t_i(\eps)$ is the meeting time of $\gamma^{\eps}$ with an edge of the tiling. We must make sure that we can define such smooth functions and in particular that for $\eps>0$ small enough the geodesics $\gamma^\eps$ always meet the same edges and in the same order.

In the next lemma we construct the function parametrizing the intersection time of the variation through geodesics with a non tangential edge of the tiling. If $e$ is an edge of the tiling, there is a diffeomorphism $\sigma_e : [0,1]\to M$ parametrizing $e$ by definition of a regular tiling.

\begin{lem}\label{lem:parametrizations nontangential edges}
    Suppose that $\gamma$ is simple and let $e$ be a non tangential edge. Suppose $e$ intersects $\gamma$ at a point $p=\gamma(t_0)=\sigma_e(s_0)$. Then there exist $\eps_p\in(0,\eps_0)$, $U_p\times  V_p$ neighborhood of $(s_0,t_0)$ and smooth functions $\sps,\tps : (-\eps_p,\eps_p)\to[a,b]$ such that
    \begin{equation*}\label{eqn:time reduction}
        U_p=\sps((-\eps_p,\eps_p)) 
    \end{equation*}
    and for all $\eps\in(-\eps_p,\eps_p)$ and for all $(s,\,t)\in U_p\times  V_p$ we have
    \begin{equation*}\label{eqn:IFT}
        \sigma_e(s)=\Gamma(\eps,t) \Leftrightarrow (s,t)=(\sps(\eps),\tps(\eps))
    \end{equation*}
\end{lem}

\begin{proof}
    For the sake of the implicit function theorem, if $s_0 = 0$ or $1$ then extend smoothly $\sigma_e$ to a open interval containing $[0,1]$ so that we can always assume that $\sigma_e$ is defined in a neighborhood of $s_0$. For the same reason, extend $\gamma$ to a larger segment $[a,b]$ if $p\in\bdry$.
    
    Define in a neighborhood of $(0,s_0,t_0)$ the smooth function  $F(\eps,s,t)=\sigma_e(s)-\Gamma(\eps,t)$ and note that
    \begin{equation}
        F(0,s_0,t_0)=0
    \end{equation}
    and 
    \begin{equation}
    	D_{(s_0,t_0)}F(0,\cdot,\cdot)= (\dot \sigma_e(s_0)\,\,|\,-\dot\gamma (t_0)),
    \end{equation}
    where we described the matrix by its column vectors. Therefore, $D_{(s_0,t_0)}F$ is invertible since $e$ is non tangential. Thus by the implicit function theorem there exist $\eps_p\in(0,\eps_0)$, $U_p\times  V_p$ neighborhood of $(s_0,t_0)$ and smooth functions $\sps,\tps : (-\eps_p,\eps_p)\to[a,b]$ such that for all $\eps\in(-\eps_p,\eps_p)$ and for all $(s,\,t)\in U_p\times  V_p$ we have :
    \begin{equation}\label{eqn:IFT}
        \sigma_e(s)=\Gamma(\eps,t) \Leftrightarrow (s,t)=(\sps(\eps),\tps(\eps)).
    \end{equation}
    Moreover, differentiating the equality $F(\eps,\sps(\eps),\tps(\eps))=0$ we get,
    \begin{equation}
        (\sps'(0),\tps'(0))^T=-D_{(s_0,t_0)}F^{-1}\partial_\eps F(0,s_0,t_0).
    \end{equation}
    And if we denote $J_0=\|J(t_0)\|$ and $\dot{\sigma}_e(s_0)=(u_1,u_2)$ in the basis $(\dot{\gamma}(t_0),J(t_0)/J_0)$ of $T_pM$ we have
    \begin{equation}\label{eqn:derivative s t}
        (\sps'(0),\tps'(0))=(J_0/u_2,J_0u_1/u_2).
    \end{equation}
    
    To prove the second equation of the lemma, since $\sps'(0)\ne0$ by \eqref{eqn:derivative s t}, $\sps$ is a homeomorphism around $s_0$ by the bijection theorem and therefore one can choose $\eps_1$ so that $\sps((-\eps_1,\eps_1))\subset U_p$. Then simply reduce $U_p$ to have $U_p=\sps((-\eps_1,\eps_1))$ which is actually open.
\end{proof}

To make sure that we do only finitely many operations (e.g. restrictions of neighborhoods) we also show the following lemma.

\begin{lem}\label{lem:finite number of meetings}
    There is only a finite number of points where $\gamma$ meets non tangential edges.
\end{lem}

\begin{proof}
    Let $e$ be a non tangential edge and suppose that there exist infinitely many distinct $t_n\in [a,b]$ and $s_n\in [0,1]$ such that $\sigma_e(s_n)=\gamma(t_n)$ for every $n$. Then by compactness of both intervals (extracting if necessary) we may assume that $t_n \to t_0$ and $s_n\to s_0$. By continuity of both maps we have that $\sigma_e(s_0)=\gamma(t_0)=:p$ and we can apply \cref{lem:parametrizations nontangential edges} to parametrize this intersection near $p$.
    
    For $n$ big enough, $(s_n,t_n)\in U_p\times V_p$ and by the parametrization given by \eqref{eqn:IFT} $(s_n,t_n)=(s_0,t_0)$ which is a contradiction since we chose them distinct.
    
    Each non tangential edge has therefore a finite number of meeting points with $\gamma$ and since there is a finite number of edges in the tiling, there is only a finite number of points where $\gamma$ meets an edge of the tiling.
\end{proof}

We may now start the proof.

\begin{proof}[Proof of the \cref{lem:integral if nontangential nonvanishing and simple}]
    
    First remark that \cite[Lemma~7.3]{KS} shows that since $\gamma$ is injective then $\Gamma$ is a diffeomorphism for $\eps_0$ small enough since $\Gamma$ is a normal variation through geodesics and $J$ never vanishes by \ref{ass:nevervanishing}.
    
    To prevent the edges from meeting in $\tubb[\eps_0]$ let us first reduce $\eps_0$ such that $\tub[\eps_0]$ contains only the vertices which are on $\gamma$. It is possible because $\Gamma$ is continuous, the time interval is compact and
    there is only a finite number of vertices.
    
    Denote $E$ the set of non tangential edges $e$ meeting $\gamma$ and intersecting $\tubb[\eps_0]$. For every edge $e\in E$ and apply \cref{lem:parametrizations nontangential edges} at each intersection point $p$ of $e$ with $\gamma$ to find the neighborhoods $U_p$ and $V_p$, $\eps_p>0$ and the smooth functions $s_{e,p}$ and $t_{e,p}$. 
    
    Define $\eps_1=\min(\min_{(e,p)}\eps_p,\eps_0)$ where the second minimum is taken over all couple $(e,p)$ with $e\in E$ and $p$ meeting points of $\gamma$ with $e$. There are a finite number of meeting points by \cref{lem:finite number of meetings}, so $\eps_1>0$. We can also reduce the neighborhoods $U_p$ if necessary so that for two meeting points $p\ne p'$ of $e$ with $\gamma$, $U_p\,\cap\,U_{p'}=\varnothing$. 
    
    
    Now we make sure that all the portions of edges that intersect $\tubb$ are parametrized by reducing $\tubb$ enough so that only the parametrized portions are left. For an edge $e\in E$, define $K_e=\sigma_e([0,1]\setminus\bigcup_p\,U_p)$ where the union is taken over the meeting points $p$ of $e$ with $\gamma$. $K_e$ is compact and $d(K_e,\gamma)>0$ so one can choose $\eps_e$ small enough such that $\tub[\eps_e]\subset K_e^C$. Reducing $\eps_1\leq\min_{e\in E}\eps_e$ if necessary, we then have that for $e\in E$,
    \begin{equation}\label{eqn:full parametrization}
        \sigma_e(s)=\Gamma(\eps,t) \Leftrightarrow \exists!\, p,\, (s,t)=(\sps(\eps),\tps(\eps)).
    \end{equation}
    In fact, if $\sigma_e(s)=\Gamma(\eps,t)$ then $s\in U_p$ for a unique meeting point $p$ by the previous reduction
    (unique since the sets $U_p$ do not overlap), by the second equation of \eqref{eqn:time reduction} $s=\sps(\tilde\eps)$, and $\tilde\eps=\eps$ since $\Gamma$ is a diffeomorphism. Moreover, $\Gamma(\eps,\tps(\eps))=\sigma_e(\sps(\eps))$ by \cref{lem:parametrizations nontangential edges} thus $t=\tps(\eps)$ since $\Gamma$ is a diffeomorphism.
    
    Now that the parametrization is done, one only needs to check that we can actually order the time parametrizations. 
    First, enumerate in any order the set $\{ \tps,\,e\in E,\,p\text{ meeting point}\}$ (finite) by $\{t_i\}_{i\in\libr1,N+1\ribr}$. To show that we can order the functions $t_i$ such that 
    \begin{equation}
        t_1<...<t_{N+1} \text{ on } (0,\eps_1),
    \end{equation}
    we apply the intermediate value theorem to the continuous functions $t_i$ on the interval $(0,\eps_1)$ since they cannot meet :
    \begin{equation}\label{eqn:times does not meet}
        \forall (e,p)\ne (e',p'),\, \forall \eps\in (0,\eps_1),\, \tps(\eps)\ne \tps[e',p'](\eps).
    \end{equation}
    In fact, suppose $\tps(\eps)=\tps[e',p'](\eps)$ for $\eps\in(0,\eps_1)$. Denote by $q:=\Gamma(\eps,\tps(\eps))=\Gamma(\eps,\tps[e',p'](\eps))$. By \eqref{eqn:full parametrization}, $q\in\sigma_e(U_p)\cap\,\sigma_{e'}(U_{p'})$. On the first hand if $e\ne e'$, $q$ is a vertex in $\tubb$ which is excluded by the second step of the proof. On the other hand, if $e=e'$ since the sets $U_p$ do not overlap, the sets $\sigma_e(U_p)$ do not either and $p=p'$.
    
    There exist $a_i\in[a,b]$ values of $f$ such that
    \begin{equation}\label{eqn:f constant}
        \forall i\in\libr1,N\ribr,\forall\eps\in(0,\eps_1), f\circ\gamma^\eps\equiv a_i \text{ on } (t_{i}(\eps),t_{i+1}(\eps)).
    \end{equation}
    In fact, for $\eps\in(0,\eps_1)$ if $t\in(t_{i}(\eps),t_{i+1}(\eps))$ then $\gamma^{\eps}$ does never touch an edge of the tiling and therefore $f\circ\gamma^{\eps}$ is constant. 
    Indeed, since $p\in\inter{M}$, boundary edges around ending points are in $E$ and by assumption, tangential geodesic edges are segments of $\gamma$. So if
    $\gamma^{\eps}(t)=\Gamma(\eps,t)=\sigma_e(s)$ with $\eps>0$ for an edge $e$ then $e\in E$ and by \eqref{eqn:full parametrization} $t=\tps(\eps)=t_j(\eps)$ for an index $j$.
    
\end{proof}

\subsection{Geometric assumptions : simple manifolds with geodesic tilings}
We must now give geometric assumptions on the manifold and the tiling implying \ref{ass:nontangential} and \ref{ass:nevervanishing} of \cref{lem:integral if nontangential nonvanishing}.

\begin{defin}[Geodesic tiling]\label{def:geodesic tiling}
    We say that a regular tiling of a Riemannian 2-dimensional manifold with boundary is a geodesic tiling if the edges of the tiling are either geodesic segments or boundary segments. 
\end{defin}
Geodesic tilings will be the most natural tilings satisfying the assumption \ref{ass:nontangential} of \cref{lem:integral if nontangential nonvanishing}
merely by the uniqueness of a geodesic with given starting point and velocity.

For the assumption \ref{ass:nevervanishing}, we show that on a simple manifold, one can always find a non vanishing Jacobi field with a given initial value at a point.
\begin{lem}\label{lem:existence of jacobi simple}
    Let $p\in M$ and $v\in T_pM \setminus \{0\}$. Let also $\gamma$ be a maximal geodesic starting at $p$.
    If $M$ is simple then there exist a Jacobi field $J$ along $\gamma$ which is never vanishing and such that $J(0)=v$.
\end{lem}

\begin{proof}
    Since $M$ is simple, we can extend $M$ to $\tild{M}$ a simple manifold which contains a neighborhood of $M$. Extend also the geodesic $\gamma$ to $\tild{M}$. Let $t_0$ such that $q:=\gamma(t_0)\in\tild{M}\setminus M$. Since $\tild{M}$ is simple, $q$ and $p$ are not conjugate points and therefore the two-point boundary problem for Jacobi fields is uniquely solvable (see \cite[Exercise~10.2]{RM}) or in other words there exists a unique Jacobi field $J$ along $\gamma_\theta$ such that $J(t_0)=0$ and $J(0)=v$. Since $\tild{M}$ is simple and $J(0)=v\ne0$, $J$ does not vanish except on $t_0$ and therefore does not vanish in $M$. Replacing $J$ by its restriction to $M$ we proved the lemma.
\end{proof}

\subsection{Proof of \cref{lem:reconstruction simplegeod iteration step}}
\label{subsctn:proof reconstruction simplegeod}



The following proof uses only the Jacobi fields of the variations that we use in the general case when $p\in\inter{M}$.

\begin{proof}[Proof of \cref{lem:reconstruction simplegeod iteration step}]
    The structure is the same as in \cref{lem:reconstruction noass iteration step} and we take also the same notations.
    
    Again, the simplices of the first type intersect $\{\phi>c\}$.
    
    Here, the tangent tiles intersect also $\{\phi>c\}$. In fact, if $w\in H_0\cap C_p\Delta$ since tangent cones have non empty interior and $C_p\Delta\cap H_- =\varnothing$, $w$ must be the tangent vector to an edge of $\Delta$. Since the tiling is geodesic, this edge is a geodesic with velocity $w$ at $p$ tangential to $\Sigma$. Since $\Sigma$ is strictly convex, this edge and therefore $\Delta$ also intersect $\{\phi>c\}$.

    The goal is thus to use the \cref{lem:corner reconstruction} to reconstruct the unknown values of $f$ on tiles which do not intersect $\{\varphi>c\}$ which are among the corner tiles. If there are corner tiles, $p$ must be a vertex.
    
    To define the variation through geodesics, take the same notations as in \cref{subsctn:corner variation}.
    Take $J_\theta$ a non vanishing normal Jacobi field along $\gamma_\theta$ such that $J_\theta(0)=w_{\theta+\pi/2}$ given by the \cref{lem:existence of jacobi simple} and denote $\Gamma_\theta : (-\eps_0,\eps_0)\times[a,b] \to M$ any variation through geodesics starting from $\gamma_\theta$ with Jacobi field $J_\theta$ and $\gamma^\eps_\theta=\Gamma(\eps,\cdot)$.

    By the \cref{lem:integral if nontangential nonvanishing}, there exist $\eps_1>0$, $N\in\N$, smooth functions $t_i : [0,\eps_1)\to[a,b]$ for $i\in\libr1,N+1\ribr$ and $(a_i)_{i\in\libr1,N\ribr}$ values of $f$ associated to tiles $\Delta_i$ such that,
    \begin{equation}\label{eqn:use of lemma for reconstruction}
        \forall \eps\in(0,\eps_1), \int_{\gamma^\eps_\theta}{f\,ds}=\sum_{i=1}^{N}{a_i(t_{i+1}(\eps)-t_{i}(\eps))}.
    \end{equation}
    
    Define $\mathcal{I}$ as the set of indices $i$ such that $t_i$ and $t_{i+1}$ are the meeting time of an edge which is not in $\mathcal{C}$. Then,
    \begin{equation}\label{eqn:integral in corner}
        \int_{\gamma^\eps_\theta\cap C}{f\,ds}=\int_{\gamma^\eps_\theta}{f\,ds}-\sum_{i\in\mathcal{I}}{a_i(t_{i+1}(\eps)-t_{i}(\eps))}.
    \end{equation}

    Let now $B(p,r)$ be the ball defined in \cref{sctn:corner}. By strict convexity of $\Sigma$, \cref{lem:geodesics in foliated manifolds} ensures that for small enough parameters $(\theta,\eps)$, $\gamma_\theta$ stays in $\{\phi>c\}$ outside $B(p,r)$. In $B(p,r)$, the values $a_i$ of $f$ with $i\in\mathcal{I}$ are known because the associated tiles are of first type or tangent tiles. Therefore, the values $a_i$ of $f$ for $i\in\mathcal{I}$ are known.
    
    Since the right hand side of \eqref{eqn:integral in corner} is known, we may use \cref{lem:THE reconstruction in a corner} to conclude. This last lemma uses only the initial derivatives in $\eps$ of the left hand side. Since the initial derivatives of the meeting times $t_i$ can be expressed thanks to the Jacobi field of the variation by \eqref{eqn:derivative s t}, the knowledge of the Jacobi field and the ray transform of $f$ allows us to reconstruct $f$.
\end{proof}

\begin{rk}
	This proof shows that we may use only vertices of the tiling in the reconstruction when $M$ is simple and the tiling is geodesic.
\end{rk}

    \section{Stability}\label{sctn:stability}
        Let us now discuss the stability of the reconstruction method presented in the previous sections. As we said in the introduction, the goal of this section is to bound from above a norm of $f$ by a suitable norm of its ray transform.

\subsection{Regularity of the ray transform}\label{sctn:regularity}
Since the reconstruction method relies on taking derivatives of some smooth parts of integrals, a natural norm one could choose for the ray transform of a function would be linked to the regularity of the function : $C^k$-norm, Hölder norm, ... but here the function $\ray$ is not even continuous in general because $f$ is not continuous as shown by the example below. 

\begin{ex}\label{ex:non continuity circle}
    In $\R^2$ with the Euclidean metric, define $M=\bar D(0,1)$ the closed unit disk which is a compact manifold with strictly convex boundary for the induced metric. Take the tiling given by the cut in four equal radial parts by the canonical axes and let $f$ be equal to $1$ on each tile. Consider the lines $\ell_\eps$ defined by $y=\eps$. Then $\int_{\ell_\eps}{f\,ds}$ is not continuous at $\eps=0$ since $f$ is equal to $0$ on the edges. This lack of continuity here is due to the geodesics being entirely along an edge when $\eps=0$, exactly as in \cref{lem:integral if nontangential nonvanishing}.
\end{ex}

Since the ray transform is not even continuous everywhere, it is hopeless to put a norm linked with derivatives on the whole manifold. One may think after the previous example that $\ray$ could admit smooth extensions so that a suitable norm on the extension could bound the values of $f$. This is not true as shown by the following example.
\begin{ex}
    Let us study a simple case in $\R^2$ with $M=\bar D(1,1)$ with a tiling having $0$ as a vertex. To reconstruct $f$ in the corner at $p=0$ we must compute the integrals of $f$ over $\gamma_\theta^\eps$ and more specifically its derivatives in $\eps$ and $\theta$ at $0$, thus also of the positive ending time $t^+(\theta,\eps)$ of $\gamma_\theta^\eps$. However
        \begin{equation}
            t^+(\theta,\eps)=\sin(\theta)+\sqrt{\sin^2(\theta)+\eps\cos(\theta)-\eps^2}
        \end{equation}
        which is $C^1$ near $0$ for $\theta>0$ and $\eps>0$ but not at $(0,0)$. There is therefore no reason for the integrals to be $C^1$ at $(0,0)$. Moreover, here
	\begin{equation}
		\partial_\eps t^+(\theta,0)=\frac{1}{2\tan(\theta)},
	\end{equation}
	so $\partial_\eps t^+(\theta,0)$ is not even bounded for $\theta$ near $0$ so it does not admit any $C^1$ extension.
\end{ex}

\begin{rk}\label{rk:tangential singularities}
The two previous examples show the reason why $\ray$ may have singularities. A singularity may appear when $f$ is integrated along geodesics becoming tangent to a change of values of $f$.
\end{rk}

At each step of the reconstruction, the geodesics used in the proofs might be tangential to a change of values of $f$. Therefore, we will not be able to state a stability estimate involving some norm of $\ray$ but we will rather use norms of the restrictions of the integrals of $f$ along portions of geodesics considered in the reconstruction because these parts of integrals are smooth with respect to the parameters. These restricted integrals are not initially given in the problem, but they can be computed from the knowledge of the ray transform, the metric and the tiling. 
	
To formulate an estimate controlling all the values of $f$ at once, we will state stability estimates at each step of the reconstruction and then compile them to get a global one. Following the reconstruction exposed in \cref{sctn:iteration}, let $p\in M$ be a point around which we wish to reconstruct $f$, lying on a level set $\{\phi=c\}$ such that $f$ is known in the set $\{\phi>c\}$. Let $a$ be the value of $f$ on a tile $\Delta$ containing $p$. The stability estimate will depend on the type of $\Delta$. 

\subsection{Tangent tiles}

Suppose here that $\Delta$ is a tangent tile. If $\Delta$ intersects $\{\phi>c\}$ then $a$ is already known. If $\Delta$ does not intersect $\{\phi>c\}$, then we use the geodesics $\gamma^\eps$ defined in \cref{subsctn:proof reconstruction noass} to recover $a$.

In this last case, denoting $\kappa_\Delta(p)$ the curvature of the edge $e$ of $\Delta$ containing $p$ at this point and $\kappa(p)$ the curvature of the hypersurface $\{\phi=c\}$ at $p$ we have $\kappa_\Delta(p)>\kappa(p)$ so in particular $\kappa_\Delta(p)>0$.  Thus, a straightforward adaptation of the proof of \cref{lem:side tiles reconstruction} and \cref{lem:time side tiles reconstruction} with the hypersurface $e$ instead of $\bdry$ near $p$ gives the existence of $\eps_1>0$ so that we have at $p$,
    \begin{equation}\label{eqn:I_delta}
        \forall \eps\in[0,\eps_1], I_\Delta(\eps) := \int_{\gamma^\eps\cap\Delta}{f\,ds}=a\,t(\eps)
    \end{equation}
    where $t(\eps)$ is the smooth positive time at which $\gamma^\eps$ meets $e$ which satisfies $t'(0)=2/\kappa_\Delta(p)$. This yields the following lemma.
\begin{lem}
        If $\Delta$ is a tangent tile, then
		\begin{equation*}
		 |a| \leq \frac{\kappa_\Delta(p)}{2}\, \| I_\Delta \|_{C^1}
		\end{equation*}
		 where $\kappa_\Delta(p)>0$ is the curvature of the edge of $\Delta$ containing $p$ at this point.
    \end{lem}

The previous lemma is worth to be restated in the particular case $p\in\bdry$, since in this case the edge of $\Delta$ containing $p$ is all along the boundary. The function $I_\Delta$ is then the full ray transform of $f$ and $\kappa_\Delta(p)$ is equal to $\kappa(p)$ the curvature of $\bdry$ at $p$. Since the set $V_\partial$ of vertices on the boundary is finite, we may choose $\eps_0>0$ so that \cref{eqn:I_delta} holds for each $p\in V_\partial$. If we denote $\nu(p)$ the inward pointing normal to the boundary at $p$ and if we define $N_{\eps_0} := \{ (p,v)\in V_\partial \times T_pM \,|\, \inner{v}{\nu(p)} \leq\eps_0\}$, we then have that $\ray$ is well defined and smooth on $N_{\eps_0}$.

\begin{cor}\label{cor:tangent tile on boundary}
	If $p\in\bdry$ and $\Delta$ is a tangent tile, then
		\begin{equation*}
		 |a| \leq \frac{\kappa(p)}{2}\, \| \ray_{|N_{\eps_0}} \|_{C^1}
		\end{equation*}
		 where $\kappa(p)>0$ is the curvature of $\bdry$ at $p$.
\end{cor}

\begin{rk}
	Following the \cref{rk:side tiles noncanonical method}, we could have recovered $f$ on the tangent tiles at the boundary using a geodesic for a fixed $\eps>0$ small enough and found a stability result for the $L^\infty$ norm of the ray transform but the constant involved in the estimate would depend on $\eps$.
\end{rk}

\subsection{Corner tiles}

Suppose now that $\Delta$ is a corner tile. If there are corner tiles containing $p$, then $p$ must be a vertex of the tiling. As always the meeting times with non tangential edges are parametrized by smooth times by the implicit function theorem. This is the reason why even in the general case, the integrals of $f$ restricted in the corner

\begin{equation}
	I_C(\theta,\eps):=\int_{\gamma_{\theta}^{\eps}\cap\,C}{f\,ds}
\end{equation}
where the geodesics $\gamma^\eps_\theta$ are defined in \cref{subsctn:middle tiles} are smooth for small enough parameters. The values on the tiles of $C$ are then given by \cref{eqn:values equal matrix derivatives} and if we denote here $A_C$ the matrix involved in this equation and $m_C$ the number ot tiles in the corner, we have the following result. 

\begin{lem}
	If $\Delta$ is a corner tile then $|a| \leq C_C\|I_C\|_{C^{m_C}}$ with $m_C$ the number of tiles in the corner $C$ and $C_C \propto \|A_C^{-1}\|_\infty$.
\end{lem}

\begin{rk}
	The proportionality coefficient in $C_C\propto \|A_C^{-1}\|_\infty$ comes from the derivation of the function defined by \cref{eq:def of F} and could thus be computed explicitly.
\end{rk}

This time, even in the case $p\in\bdry$, we may not state a stability estimate using $\ray$. In fact, when integrating along geodesics almost tangential to the boundary a lack of regularity comes from the ending times, see \cref{rk:sqrt singularity at boundary}. 

\subsection{Weak global estimate}

We may now compile all the previous estimates to get a global one. Since the number of tiles is finite, there is a finite number of points considered in the reconstruction. Denote $T$ the set of tiles which appear as tangent tiles in the reconstruction. For each vertex $p$ in the manifold we may define a corner $C$ with respect to the hypersurface $\{\phi = \phi(p)\}$ as in \cref{subsctn:corner def}. Denote then $\mathcal{C}$ the finite set of all corners defined that way in the manifold and for $C\in\mathcal{C}$ denote $m_C$ the number of tiles in $C$ and $A_C$ the matrix associated to $C$ defined in \cref{subsctn:corner reconstruction}. We may now state our global stability statement in the general case which directly follows from the previous sections.

\begin{prop}
	If $f$ is a piecewise constant function on a two-dimensional compact nontrapping Riemannian manifold with strictly convex boundary and if $\{a_i\}_{i\in I}$ are the values of $f$, then
	\begin{equation}
		\max_{i\in I} |a_i| \leq C_t \max_{\Delta\in T}\|I_\Delta\|_{C^1} + C_c \max_{C\in\mathcal{C}} ||I_C||_{C^{m_C}}
	\end{equation}
	where
	\begin{equation}
		C_t = \max_{\Delta\in T} \kappa_\Delta(p) \quad \text{ and } \quad C_c \propto \max_{C\in\mathcal{C}} ||A_C^{-1}||_\infty.
	\end{equation}
\end{prop}

We may state a more specific proposition if $M$ is simple and the tiling is geodesic since there are no tangent tiles but on the boundary and on these tile we may use a norm on the full ray transform by \cref{cor:tangent tile on boundary}.

\begin{prop}
	If $f$ is a piecewise constant function on a geodesic tiling of a simple two-dimensional compact nontrapping Riemannian manifold with strictly convex boundary and if $\{a_i\}_{i\in I}$ are the values of $f$, then
	\begin{equation}
		\max_{i\in I} |a_i| \leq C'_t \|\ray_{|N_{\eps_0}}\|_{C^1} + C_c \max_{C\in\mathcal{C}} ||I_C||_{C^{m_C}}
	\end{equation}
	where
	\begin{equation}
		C'_t = \max_{p\in\bdry} \kappa(p) \quad \text{ and } \quad C_c \propto \max_{C\in\mathcal{C}} ||A_C^{-1}||_\infty.
	\end{equation}
\end{prop}
    
     \appendix
     \section*{Appendix}
     \renewcommand{\thesection}{A} 
\setcounter{section}{0}
\setcounter{thm}{0}
\setcounter{equation}{0}
\setcounter{figure}{0}
\setcounter{table}{0}


%
%
%

\subsection{Geodesics in foliated manifolds}

For $q\in M$ and $w\in T_qM$, denote $\gamma_{q,w}: I_{q,w}\to M$ the maximal geodesic with starting point $q$ and initial velocity $w$. 

\begin{lem}\label{lem:geodesics in foliated manifolds}
	Let $M$ be a compact nontrapping 2-dimensional Riemannian manifold with strictly convex boundary.
    Let $\phi : M \to \R$ be a strictly convex foliation of $M$ and $\min\varphi< c\leq\max\varphi$. Denote $\Sigma=\{\varphi=c\}$ and let $p\in\Sigma$ and $\omega\in T_p\Sigma$. Then for all neighborhood $U$ of $p$, there exist a neighborhood $V$ of $p$ and $W$ of $\omega$ such that for all $q\in V$ and $w\in W$, there exist $t_0,t_1\in I_{q,w}$ such that $\gamma_{q,w}$ meets $\Sigma$ only at $t_0$ and $t_1$ in $U$ and stays in $\{\phi>c\}$ outside $[t_0,t_1]$. 
\end{lem}

\begin{proof}
    By strict convexity of $\Sigma$ at $p$, there exist two neighborhoods $V$ of $p$ and $W$ of $\omega$ such that $\gamma_{q,w}$ meets $\Sigma$ in $U$ at two times $t_0\leq t_1\in I_{q,w}$ equal if and only if $q=p$ and $w=\omega$. 
    
    Therefore $\zeta:=\varphi\circ\gamma_{q,w}$ is equal to $c$ in $U$ either
    \begin{itemize}
        \item at the critical point $t_0=t_1=0$ if $(q,w)=(p,\omega)$,
        \item or else at $t_0<t_1$ and therefore has a critical point in between by Rolle's theorem.
    \end{itemize}
    Outside of these times, $\gamma$ stays in $\{\phi>c\}$. In fact, by \cite[Theorem~2.2]{UD94} the function $\zeta$ is strictly convex and so $\zeta$ has a strictly increasing derivative. Since its derivative vanishes at a point in $[t_0,t_1]$ we have that $\zeta$ is strictly decreasing before $t_0$ and strictly increasing after $t_1$. Since $\zeta(t_0)=\zeta(t_1)=c$. We can conclude from there that $\zeta>c$ outside the meeting times (possibly equal).  
    
\end{proof}

\subsection{Jerk of the boundary}
\label{subsctn:appendix jerk}

In this section we prove \cref{lem:parametrization boundary}. Take the same framework as in this lemma : $p\in \bdry$, $\nu$ the inward pointing normal to the boundary at $p$ and $\omega \in T_p\bdry$ a unit vector so that $(\omega,\nu)$ is an orthonormal basis of $T_pM$. Denote $(\phi,U)$ a normal coordinate chart of $p$ associated to this basis.

\begin{proof}[Proof of \cref{lem:parametrization boundary}]
	Let $(\psi,V)$ be a submanifold chart for $\bdry$ in $M$, such that $\bdry\cap V = \psi^{-1}(\R\times\{0\})$ and $d_p\psi(\omega)=(1,0)$. If we denote $g(x)=(w(x),z(x))=\phi \circ \psi^{-1}(x,0)$ then $w,z : \R \to \R$ are smooth near $0$ and $(w'(0),z'(0))=d_p\phi \circ d_{(0,0)}\psi^{-1}(1,0)=d_p\phi(\omega)=(1,0)$. Therefore, $w$ is a diffeomorphism near $0$. 
	
	Now replacing $\psi$ by the submanifold chart $(w\times Id)\circ \psi$ we get $g(x)=(x,h(x))$ where $h=z\circ w^{-1}$ is smooth near $0$. For $q$ near $p$ in the normal chart we have 
	\begin{align}
		\phi(q)=(x,y)\in \bdry &\Leftrightarrow \exists u\in\R, \phi^{-1}(x,y)=q=\psi^{-1}(u,0) \\
		&\Leftrightarrow \exists u\in\R, (x,y)=g(u)=(u,h(u)) \\
		&\Leftrightarrow y=h(x).
	\end{align}
	
	The only thing left to prove the Taylor's expansion of $h$ is that $h'(0)=0$ and $h''(0)=\kappa$. The first equality follows from $h'(0)=z'(0)w'(0)=0$. To prove the second consider a unit speed geodesic of the boundary $\bdry$, $\gamma(t)=(x(t),y(t))$ starting from $p$ with initial velocity $\omega$. We have that the initial covariant derivative along $\gamma$ of its velocity vector is given by 
	\begin{equation}\label{eqn:cov derivative}
	 D_t\dot\gamma (0) = \kappa\,\nu = (0,\kappa)
	 \end{equation} 
	by definition of the curvature $\kappa$. We have near $0$ that $y(t)=f(x(t))$ so $\ddot y(t)= \dot x(t)^2 h''(t) + \ddot x(t) h'(t)$. Since the Christoffel's symbols vanish at $p$ in the normal coordinate chart we have $\dot x(0) = 1$ and $D_t\dot\gamma (0)=(\ddot x(0),\ddot y(0))$. We may conclude by \eqref{eqn:cov derivative} that $\kappa=\ddot y(0) = h''(0)$.
\end{proof}
    
    \bibliographystyle{siam}
    \bibliography{mybib}
    
\end{document}